\documentclass[10pt]{article} 

\usepackage{xspace}
\usepackage{url}
\usepackage{mathtools}
\usepackage{amssymb}
\usepackage{amsthm}
\usepackage{empheq}
\usepackage{latexsym}
\usepackage{enumitem}
\usepackage{eurosym}
\usepackage{dsfont}
\usepackage{appendix}
\usepackage{color} 
\usepackage[unicode,hypertexnames=false]{hyperref}
\usepackage{frcursive}
\usepackage[utf8]{inputenc}
\usepackage[T1]{fontenc}
\usepackage{geometry}
\usepackage{multirow}
\usepackage{comment}
\usepackage{todonotes}
\usepackage{lmodern}
\usepackage{anyfontsize}
\usepackage{stmaryrd}
\usepackage{natbib}
\usepackage{cleveref}
\usepackage[english]{babel}
\usepackage[english=british]{csquotes}
\usepackage{mathtools}
\usepackage{accents}
\usepackage{color}
\usepackage{bm}
\usepackage{subfig}
\usepackage{verbatim}
\usepackage{url}
\usepackage{aliascnt}
\usepackage{microtype}

\setcitestyle{numbers,open={[},close={]}}

\newcommand\citeayn[2][]{\citeauthor*{#2} (\citeyear{#2}) \cite[#1]{#2}}
\newcommand\citeny[2][]{\textnormal{\cite{#2}~(\citeyear{#2})}}

\renewcommand{\emph}[1]{\textit{#1}}

\usepackage[english]{babel}
\usepackage[english=british]{csquotes}

\DeclareMathOperator*{\argmax}{arg\;max}
\DeclareMathOperator*{\argmin}{arg\;min}

\definecolor{red}{rgb}{0.7,0.15,0.15}
\definecolor{green}{rgb}{0,0.5,0}
\definecolor{blue}{rgb}{0,0,0.7}
\hypersetup{colorlinks, linkcolor={blue},citecolor={green}, urlcolor={red}}

\usepackage{xcolor,pifont}

\makeatletter
\@addtoreset{equation}{section}

\makeatother

\newtheorem{theorem}{Theorem}[section]

\newaliascnt{assumption}{theorem}
\newtheorem{assumption}[assumption]{Assumption}
\aliascntresetthe{assumption}

\newaliascnt{proposition}{theorem}
\newtheorem{proposition}[proposition]{Proposition}
\aliascntresetthe{proposition}

\newaliascnt{definition}{theorem}
\newtheorem{definition}[definition]{Definition}
\aliascntresetthe{definition}

\newaliascnt{lemma}{theorem}
\newtheorem{lemma}[lemma]{Lemma}
\aliascntresetthe{lemma}

\newaliascnt{example}{theorem}

\aliascntresetthe{example}

\newaliascnt{corollary}{theorem}

\aliascntresetthe{corollary}

\newaliascnt{remark}{theorem}
\newtheorem{remark}[remark]{Remark}
\aliascntresetthe{remark}

\newaliascnt{condition}{theorem}
\newtheorem{condition}[condition]{Condition}
\aliascntresetthe{condition}

\crefname{theorem}{theorem}{theorems}
\Crefname{theorem}{Theorem}{Theorems}

\crefname{assumption}{assumption}{assumptions}
\Crefname{assumption}{Assumption}{Assumptions}

\crefname{proposition}{proposition}{propositions}
\Crefname{proposition}{Proposition}{Propositions}

\crefname{definition}{definition}{definitions}
\Crefname{definition}{Definition}{Definitions}

\crefname{lemma}{lemma}{lemmas}
\Crefname{lemma}{Lemma}{Lemmas}

\crefname{example}{example}{examples}
\Crefname{example}{Example}{Examples}

\crefname{corollary}{corollary}{corollaries}
\Crefname{corollary}{Corollary}{Corollaries}

\crefname{remark}{remark}{remarks}
\Crefname{remark}{Remark}{Remarks}

\crefname{equation}{equation}{equations}
\Crefname{equation}{Condition}{Conditions}

\crefname{condition}{condition}{conditions}
\Crefname{condition}{Condition}{Conditions}

\DeclareUnicodeCharacter{014D}{\=o}
\def \D{\mathbb{D}}
\def \E{\mathbb{E}}
\def \F{\mathbb{F}}

\def \H{\mathbb{H}}

\def \L{\mathbb{L}}

\def \N{\mathbb{N}}

\def \P{\mathbb{P}}

\def \R{\mathbb{R}}

\def\Ac{{\cal A}}
\def\Bc{{\cal B}}
\def\Cc{{\cal C}}

\def\Fc{{\cal F}}

\def\Kc{{\cal K}}

\def\Oc{{\cal O}}
\def\Pc{{\cal P}}

\def\Rc{{\cal R}}
\def\Sc{{\cal S}}

\def\Vc{{\cal V}}
\def\Wc{{\cal W}}

\def\eps{\varepsilon}

\def\drm{\mathrm{d}}
\newcommand{\smallertext}[1]{\text{\fontsize{6}{6}\selectfont$#1$}}
\newcommand{\smalltext}[1]{\text{\fontsize{4}{4}\selectfont$#1$}}

\def\eps{\varepsilon}
\definecolor{bleudefrance}{rgb}{0.19, 0.55, 0.91}

\definecolor{darkspringgreen}{RGB}{60, 179, 113}

\definecolor{viola}{RGB}{200, 0, 255}

\title{Forcing and duality‑corrected contracts for volatility control}

\author{Alessandro {\sc Chiusolo}\thanks{Department of Operations Research and Financial Engineering (ORFE), Princeton University, USA.} \thanks{Research partially supported by the NSF grant DMS-2307736.} \and Emma {\sc Hubert}\thanks{CEREMADE, Université Paris–Dauphine, PSL Research University, France.} \footnotemark[2] \and Dylan {\sc Possama\"i}\thanks{Department of Mathematics, ETH Zürich, Switzerland. This author gratefully acknowledges support from the SNF project MINT 205121-21981.} \and Nizar {\sc Touzi}\thanks{New York University, Tandon School of Engineering, USA.}}

\date{\today}

\begin{document}
	
\maketitle

\begin{abstract}
In this paper, we revisit the construction of optimal incentives in continuous-time principal--agent problems with drift and volatility control. Originally, a general approach relying on dynamic programming and second-order backward stochastic differential equations (2BSDEs) was developed by \citeayn{cvitanic2018dynamic} to determine the optimal form of contracts in this setting. More recently, \citeayn{chiusolo2026new} proposed a BSDE-based approach by introducing an alternative `contractible-volatility' problem for the principal. In addition to the proposed new method, this work highlights that the optimality result of \cite{cvitanic2018dynamic} actually hinges on an assumption---stated below as \Cref{ass:duality}---which may not hold in general. Motivated by this, we introduce in this paper a more general class of contracts, parametrised by a function $\psi$ subject to conditions that make the contract revealing for the agent and without loss of generality for the principal. We further provide two natural specifications of $\psi$: one, inspired by the BSDE approach, yielding a forcing-type contract; the other, motivated by the 2BSDE approach, correcting the duality gap when \Cref{ass:duality} is not satisfied.
\end{abstract}

\section{Introduction}\label{sec:intro}

Principal--agent problems originate in the economic literature on optimal incentives under asymmetric information, particularly in the context of moral hazard. The foundational literature was shaped by the influential contributions of \citeauthor{mirrlees1976optimal} in \citeny{mirrlees1976optimal} and \citeny{mirrlees1999theory}, \citeayn{holmstrom1979moral}, and \citeayn{grossman1983analysis}. More precisely, in these models a principal (she) is imperfectly informed about the efforts of an agent (he), and seeks to incentivise him to undertake appropriate actions over a given period. To that end, the principal can design a contract, usually represented by a terminal payment $\xi$ at a finite time horizon $T>0$, and indexed on what she can observe. In particular, in a continuous-time stochastic framework, it is usually assumed that the agent can choose some control process unobservable by the principal, but impacting the drift and/or the volatility of an observable output, represented by a stochastic process $X$. The goal of the principal is therefore to determine an optimal contract, as a function of the trajectories of $X$, in order to maximise her expected utility, and anticipating the agent's optimal response to this contract. It is also commonly assumed that the agent will decline the contract proposed by the principal if his resulting utility falls below a given level, defined as his reservation utility, thus adding a constraint for the principal in her design problem. For classical textbook treatments of principal--agent models, see \citeayn{laffont2002theory}, \citeayn{bolton2005contract} and \citeayn{salanie2005economics}.

\medskip

Originally introduced by \citeayn{holmstrom1987aggregation}, continuous-time versions were later generalised by \citeayn{sannikov2008continuous}. His approach---coined \textit{Sannikov's trick} in the associated literature---consists in using the agent's continuation utility to determine the optimal form of terminal payments, and further consider it as an additional state variable for the principal. Since then, continuous-time contract theory has been the subject of an extensive literature, with overviews provided in the monographs by \citeayn{cvitanic2012contract} and \citeayn{sung2023contract}. In a drift control framework, this continuation utility can naturally be related to the first component of the solution to a backward stochastic differential equation (BSDE) by standard stochastic control theory. When adding the possibility for the agent to impact the volatility of the output process, this continuation utility now has to be related to a second-order backward stochastic differential equation (2BSDE), as shown by \citeayn{cvitanic2018dynamic}. More precisely, in this paper, the relevant form of contracts is shown to be indexed on the trajectory of the output $X$ through a process $Z$ chosen by the principal---as in the classical drift control case---but also on its quadratic variation $[X]$, through an additional process $\Gamma$ also chosen by the principal. Informally, while indexing the contract on the output process $X$ incentivises the agent to undertake some efforts on the drift, the indexation on its quadratic variation encourages the control of the volatility. In the end, they prove that the original principal's optimisation problem boils down to a more standard stochastic control problem, with two state variables $X$ and $\xi$ controlled by the pair $(Z,\Gamma)$. 

\medskip

Very recently, \citeayn{chiusolo2026new} presented an alternative approach to proving the optimality of the contract form introduced in \cite{cvitanic2018dynamic}, within the same general framework. This approach, which is based solely on BSDE theory, consists in defining another optimisation problem for the principal, called `\textit{first-best}', in the sense that there is no moral hazard on the quadratic variation of the output process. More precisely, while in the original problem the agent controls the drift and the volatility of the stochastic process $X$, the alternative one consists in assuming that, since the quadratic variation is observable in continuous-time, it can be directly controlled by the principal, who can in turn force the agent to achieve a specific choice.
Note that the term \textit{first-best} in principal--agent problems commonly refers to a setting without moral hazard, in which the principal directly chooses the agent's controls, as opposed to \textit{second-best}. To avoid confusion, we depart slightly from the terminology in \cite{chiusolo2026new} by using here the term \textit{contractible‑volatility} to designate this alternative problem. There, the principal thus chooses both a contract $\xi$ and the density with respect to the Lebesgue measure of the quadratic variation, later denoted $\Sigma$, to which the agent replies with an effort constrained to achieve the specific density $\Sigma$ fixed by the principal. The advantage is that this alternative problem can be solved by following \textit{Sannikov's trick}, thus relying solely on BSDEs. Then, using the contract form proposed in \cite{cvitanic2018dynamic} in the original problem, it can be shown that the principal actually achieves her \textit{contractible‑volatility} value, which is the maximum possible value, thus ensuring optimality of the contract form introduced in \cite{cvitanic2018dynamic} and equivalence between the original and the alternative problems.

\medskip

However, as highlighted in \cite{chiusolo2026new}, the equivalence between the two formulations---hence the two approaches---holds only under \Cref{ass:duality}. More importantly, the main result of \cite{cvitanic2018dynamic}, namely Theorem 3.6 establishing the optimality of the proposed contract form, may fail when this assumption does not hold. While \Cref{ass:duality} is sufficient to guarantee optimality of the original contract form and is satisfied in most examples considered in the literature, it remains crucial to investigate the existence of optimal contracts beyond this setting, in order to reach a comprehensive understanding of optimal incentives in continuous-time principal--agent problems with volatility control. From an economic perspective, when this assumption is not satisfied, it remains unclear whether the principal can still attain her value in the corresponding \textit{contractible‑volatility} alternative problem introduced in \cite{chiusolo2026new}.

\medskip

We show in this paper that, even when \Cref{ass:duality} does not hold, one can still identify an optimal form of contracts, in the sense that they are revealing for the agent and that restricting the study to this specific form is without loss of generality for the principal. The intuition behind this contract form arises from the 2BSDE methodology developed in \cite{cvitanic2018dynamic}, but the BSDE approach from \cite{chiusolo2026new} can also be used to establish its optimality. In particular, this proves that the \textit{contractible‑volatility} value---in the sense of the \emph{first-best} problem in \cite{chiusolo2026new}---can always be attained. More precisely, our study yields the following main contributions.

\medskip
  $(i)$ \textbf{Where the 2BSDE proof can fail.} We remind one of the main findings of \cite{chiusolo2026new}, namely the exact obstruction in the 2BSDE route of \cite{cvitanic2018dynamic}: in order to identify a process $\Gamma$ satisfying the identity \eqref{eq:repres_Kdot}, \textit{i.e.}
  \[
    \dot{K}_t + F_t\big(X_{\cdot\wedge t},Y_t,Z_t,\widehat{\sigma}_t^2\big)
    =
    H_t\big(X_{\cdot\wedge t},Y_t,Z_t,\Gamma_t\big) -\frac{1}{2}\mathrm{Tr}\big[\Gamma_t\widehat{\sigma}_t^2\big],
  \]
  one needs \emph{both} the equality $F=F^{\star\star}$ (where the Fenchel transform is taken with respect to the squared volatility variable) and \emph{attainment} of the infimum by a measurable selector $\gamma^{\star}$ as in \Cref{ass:duality}. Absent these, the density $\dot{K}$ of the 2BSDEs non-decreasing process cannot always be encoded by some $\Gamma$, and \cite[Theorem~3.6]{cvitanic2018dynamic} may fail.
  
  \medskip
$(ii)$ \textbf{A general $\psi$–parametrised contract class.} We introduce a contract class that specifies the density $\dot{K}$ in the 2BSDE \eqref{eq:agent_2BSDE} as $\dot{K}_t=\psi_t\big(X_{\cdot\wedge t},Y_t,Z_t,\Gamma_t,\widehat{\sigma}_t^2\big)-F_t\big(X_{\cdot\wedge t},Y_t,Z_t,\widehat{\sigma}_t^2\big)$, where for $E$ an arbitrary space,
 \[
     \inf_{S} \{ \psi_t(x,y,z,\gamma,S) - F_t(x,y,z,S) \} = 0, \;
     \min_{\gamma\in\smallertext{E}}\psi_t(x,y,z,\gamma,S)=F_t(x,y,z,S),\;\text{and}\;  \sup_{\gamma\in\smallertext{E}}\psi_t(x,y,z,\gamma,S)=+\infty.
  \]
The first condition normalises payments at the realised volatility so that the contract is \emph{revealing} (see \Cref{prop:agent}), while the second and third ones give the principal full reach over non-negative $\dot{K}$, and yields the \emph{without loss of generality} reduction (see \Cref{thm:principal}). Under \Cref{ass:duality}, $\psi$ can be chosen to recover the classical contracts of \cite{cvitanic2018dynamic}.

 \medskip
 $(iii)$ \textbf{Equality with the \textit{contractible-volatility} problem.} Let $V^{\circ}_{\smallertext{\rm P}}$ be the principal's value in the auxiliary problem where $\Sigma$ (the density of $[X]$) is selected by the principal, and $V_{\smallertext{\rm P}}$ the value in the original problem. \Cref{prop:one-to-one} gives a correspondence between feasible $(\Sigma,Z)$ and $(Z,\Gamma,\psi)$, thus ensuring that $V_{\smallertext{\rm P}}=V^{\circ}_{\smallertext{\rm P}}$, even \emph{without} \Cref{ass:duality}. 

\medskip
  $(iv)$ \textbf{Two canonical, economically transparent contracts.} The contract form we propose has the drawback of being inherently non-unique, since it is written in terms of a function $\psi$ that may be chosen arbitrarily within a very broad class. Nevertheless, we exhibit in \Cref{sec:contract_examples} two concrete and robust contract specifications: \emph{forcing contracts}, obtained by penalising observable quadratic deviations to `target' $\Sigma$; and \emph{duality–corrected contracts}, which add the explicit gap $F^{\star\star}-F$ to the original contract form in \cite{cvitanic2018dynamic}. 

  \medskip
  $(v)$ \textbf{Examples and counterexamples.} We revisit the counterexample from \cite{chiusolo2026new}, and show that our contracts \emph{attain} the \emph{contractible-volatility} value. We also remark that several explicit models in the literature satisfy \Cref{ass:duality}, but provide another counterexample within the framework of \citeayn{cvitanic2017moral}.

\medskip
The structure of the paper is as follows. In \Cref{sec:model_result}, we introduce the general model and then briefly recall the two existing approaches, developed in \cite{chiusolo2026new,cvitanic2018dynamic}. In \Cref{sec:contract_general}, we introduce a more general form of contracts, proving their optimality even when \Cref{ass:duality} is not satisfied. More precisely, \Cref{ss:main_results} states the main results, on the \textit{revealing} property of such contract form in \Cref{prop:agent} and the fact that the restriction to this form is without loss of generality from the principal point of view in \Cref{thm:principal}. The latter is then rigorously proved in \Cref{ss:proof_main}, relying on both approaches. \Cref{sec:contract_examples} introduces two canonical examples of contracts. Finally, in \Cref{sec:examples}, we review some of the existing models from the literature.

\medskip

{\footnotesize {\bf Notations.} Let $\mathbb{N}^\star\coloneqq \mathbb{N}\setminus\{0\}$ be the set of positive integers. For any $n \in \N^\star$, we will denote by $\Sc_n(\R)$ (respectively $\Sc^{\smallertext{+}}_n(\R)$) the set of symmetric (resp. symmetric positive semi-definite) $n \times n$ matrices with real entries. For any matrix $S \in \Sc_{n}^+(\R)$, we will denote by $S^{1/2}$ its $(n,n)$-dimensional symmetric square root, and by $S^{-1/2}$ its pseudo-inverse. For fixed time horizon $T > 0$, representing the maturity of the contract, we let $\Cc ( [0,T], \R^n)$ denote the set of continuous functions from $[0,T]$ to $\R^n$. The set $\Cc( [0,T], \R^n)$ will always be endowed with the topology associated to the uniform convergence on the compact $[0,T]$. For a probability space of the form $\Omega \coloneqq  \Cc ( [0,T], \R^n)$ and an associated filtration $\F$, we will have to consider processes $\psi : [0,T] \times \Omega \longrightarrow E$, for some Polish space $E$, which are $\F$-optional, \textit{i.e.} $\Oc(\F)$-measurable where $\Oc(\F)$ is the so-called optional $\sigma$-field generated by $\F$-adapted right-continuous processes. In particular, such a process $\psi$ is non-anticipative in the sense that $\psi(t,x) = \psi(t,x_{\cdot\wedge t})$, for all $t \in [0,T]$ and $x \in \Omega$. Throughout the paper, and as is customary in the probabilistic literature, we will omit to write explicitly the dependence of such processes on the canonical process $X$, and simply write $\psi_t$ instead of $\psi_t(X_{\cdot\wedge t})$, whenever no confusion can arise.
Furthermore, for any filtration $\F \coloneqq  (\Fc_{\smallertext{T}})_{t \in [0,\smallertext{T}]}$ and probability measure $\P$, we will denote by $\F^\P \coloneqq  (\Fc_{\smallertext{T}}^\P)_{t \in [0,\smallertext{T}]}$ the usual $\P$-completion of $\F$; $\F^\smallertext{+} \coloneqq  (\Fc_{\smallertext{T}}^\smallertext{+})_{t \in [0,\smallertext{T}]}$ the right limit of $\F$; $\F^{\P^\smalltext{+}} \coloneqq  (\Fc_{\smallertext{T}}^{\P^\smalltext{+}})_{t \in [0,\smallertext{T}]}$ the augmentation of $\F$ under $\P$. Finally, for technical reasons, we work under the classical ZFC set-theoretic axiom. This is a standard widely used in classical stochastic analysis theory (see \citeayn[Chapter 0.9]{dellacherie1978probabilities}). We also rely on the continuum hypothesis, which is needed to apply the aggregation result of \citeayn{nutz2012pathwise}.} 

\section{Model and previous results}\label{sec:model_result}

We first recall the framework introduced in \cite{cvitanic2018dynamic,chiusolo2026new}, and then outline their respective 2BSDE- and BSDE-based approaches.

\subsection{General framework}\label{ss:model}

Throughout the paper, fix some positive integer $d \in \N^\star$ and $T >0$ a finite time horizon, and let $X$ be the canonical process on $\Omega \coloneqq  \Cc([0,T],\R^d)$. Define the following process $\widehat \sigma^2$, taking values in $\Sc^{\smallertext{+}}_d(\R)$, 
which coincides with the density (with respect to the Lebesgue measure) of the quadratic variation of $X$,
\begin{align*}
	\widehat \sigma_t^2 \coloneqq  \limsup_{\varepsilon \to 0} \dfrac{[X]_t - [ X ]_{t-\varepsilon}}{\varepsilon}, \; t \in [0,T].
\end{align*}
Given a control process $\beta$ chosen by the agent, taking values in $B$ a subset of some Polish space, the controlled state equation is defined by the following SDE, driven by an $n$-dimensional Brownian motion $W$ for fixed $n \in \N^\star$
\begin{align}\label{eq:dynamic_SDE}
	X_t =x+\int_0^t \sigma_s(X_{\cdot\wedge s}, \beta_s) \big( \lambda_s(X_{\cdot\wedge s}, \beta_s) \drm s + \drm W_s \big), \; t \in [0,T],
\end{align}
where $\lambda : [0,T] \times \Omega \times B \longrightarrow \R^n$ and $\sigma : [0,T] \times \Omega \times B \longrightarrow \R^{d,n}$ are two bounded functions such that, for any $b \in B$, the map $[0,T]\times\Omega\ni(t,x) \longmapsto (\lambda_t(x,b), \sigma_t(x,b)) \in \R^n \times \R^{d,n}$ is $\F$--progressively measurable. When convenient, we will denote by $\mu (t,x,b) \coloneqq \sigma_t(x,b)\lambda_t(x,b)$, $(t,x,b)\in[0,T] \times \Omega \times B$, the drift coefficient of the output process.

\medskip
Admissible controls are then defined as $\F$-predictable processes, taking values in $B$, and ensuring existence of a weak solution to the above SDE. We refer to \cite{cvitanic2018dynamic} for precise details, and note whenever we follow the approach developed in \cite{chiusolo2026new}, we will additionally require weak uniqueness. 
In short, we let $\Bc$ be the set of admissible controls for the agent, and $\Pc$ the corresponding set of admissible probability measures such that, for every $\P\in\Pc$, we can find an $\F$--progressively measurable, $B$-valued process $\beta^\P$, and an $n$-dimensional $\P$--Brownian motion $W^\P$ such that
\begin{align}\label{eq:dynamic_SDE_P}
	X_t =x+\int_0^t \mu_s(X_{\cdot\wedge s}, \beta^\P_s) \drm s + \int_0^t  \sigma_s(X_{\cdot\wedge s}, \beta^\P_s) \drm W^\P_s, \; t \in [0,T].
\end{align}

Given a terminal payment $\xi\in\L^0(\R,\Fc_\smallertext{T})$, namely an $\R$-valued and $\Fc_\smallertext{T}$-measurable random variable, the optimisation problem faced by the agent is defined by
\begin{align}\label{eq:pb_agent}
&V_{\smallertext{\rm A}} (\xi) \coloneqq  \sup_{\P \in \Pc} J_{\smallertext{\rm A}} (\xi,\P), \;
	 \text{with} \;  
	J_{\smallertext{\rm A}} (\xi,\P ) \coloneqq  \E^\P \bigg[ \Kc_\smallertext{T}^\P U( \xi ) - \int_0^\smallertext{T} \Kc_s^\P c_s \big(\beta_s^{\P} \big) \mathrm{d}s \bigg], \; \P \in \Pc,
\end{align}
where
\begin{enumerate}[label=$(\roman*)$]
	\itemsep0em
	\item the Borel-measurable utility function $U : \R \longrightarrow \R$ is invertible, and we will denote by $U^{(\smalltext{-}1)}$ its inverse;
	\item the Borel-measurable cost function $c : [0,T] \times \Cc([0,T], \R^d) \times B \longrightarrow \R$ is such that, for any $b \in B$, the map $[0,T]\times\Omega\ni(t,x) \longmapsto c_t(x,b)\in\R$ is $\F$--progressively measurable, and satisfies the following integrability condition 
	\begin{align}\label{eq:integrability_cost}
		\sup_{\P\in \Pc} \E^\P \bigg[ \int_0^\smallertext{T} \big| c_s \big(\beta^{\P}_s \big) \big|^p \drm s \bigg] < + \infty, \; \text{for some} \; p >1;
	\end{align}
	\item the discount factor $\Kc^\P$ is defined for all $\P \in \Pc$ by
	\[
	\Kc_t^\P \coloneqq \exp\bigg(-\int_0^t k_s \big(\beta_s^\P \big) \drm s \bigg), \; t \in [0,T],
	\]
	where $k : [0,T] \times \Omega \times B \longrightarrow \R$ is assumed to be Borel-measurable, bounded, and such that for any $b \in B$, the map $[0,T]\times\Omega\ni(t,x)\longmapsto k_t(x,b)\in\R$ is $\F$--progressively measurable.
\end{enumerate}
To ensure that the agent's problem \eqref{eq:pb_agent} is well-defined, we require minimal integrability on the contracts as in \cite[Equation  (2.5)]{cvitanic2018dynamic}, but adapted here to our setting with utility function, namely
\begin{align}\label{eq:integrability_contract_agent}
	\sup_{\P\in \Pc} \E^\P \Big[ \big| U (\xi ) \big|^p \Big] < + \infty, \; \text{for some} \; p > 1.
	\tag{${\rm I}^p_{\smallertext{\rm A}}$}
\end{align}
In addition, we assume as usual that the agent has a reservation utility level $R_{\smallertext{\rm A}} \in \R$, which conditions his acceptance of the contract. With this in mind, we consider the following set $\Xi$ of admissible contracts
\begin{align}\label{def:contract_set}
	\Xi \coloneqq  \big\{\xi\in\L^0(\R,\Fc_{\smallertext{T}}):\text{$\xi$ satisfies \eqref{eq:integrability_contract_agent} and} \; V_{\smallertext{\rm A}} (\xi) \geq R_{\smallertext{\rm A}} \big\}.
\end{align}
Finally, we denote by $\Pc^\star(\xi)$ the set of optimal responses of the agent to a contract $\xi \in \Xi$, \textit{i.e.} 
\begin{align}\label{def:optimal_response}
	\Pc^\star(\xi) \coloneqq  \big\{ \P^\star \in \Pc : V_{\smallertext{\rm A}} (\xi) = J_{\smallertext{\rm A}} (\xi,\P^\star) \big\}.
\end{align}
Subsequently, the principal's optimisation problem is defined as follows
\begin{align}\label{eq:pb_principal}
V_{\smallertext{\rm P}} \coloneqq  \sup_{\xi \in \Xi} \sup_{\P^\smalltext{\star} \in \Pc^\smalltext{\star}(\xi)} J_{\smallertext{\rm P}} (\xi,\P^\star), \;
 \text{with} \;&
	J_{\smallertext{\rm P}} (\xi, \P ) \coloneqq  \E^\P \bigg[ \Kc^{\smallertext{\rm P},\P}_\smallertext{T}  U_{\smallertext{\rm P}} ( \xi ) - \int_0^\smallertext{T} \Kc^{\smallertext{\rm P},\P}_s c_s^{\smallertext{\rm P}} \big(\beta_s^{\P} \big) \mathrm{d}s\bigg], \; \xi \in \Xi,\; \P\in \Pc,
\end{align}
where
\begin{enumerate}[label=$(\roman*)$]
	\itemsep0em
	\item $U_{\smallertext{\rm P}} : \Omega\times\R \longrightarrow \R$ is a Borel-measurable utility function;
	\item $c^{\smallertext{\rm P}}:[0,T]\times\Omega\times B\longrightarrow \R$ is Borel-measurable, and the map $[0,T]\times\Omega\ni(t,x)\longmapsto c_t^{\smallertext{\rm P}}(x,b)\in\R$ is $\F$--progressively measurable for any $b \in  B$;
	\item the discount factor $\Kc^{\smallertext{\rm P},\P}$ is defined for all $\P\in\Pc$ by 
		\[
	\Kc^{\smallertext{\rm P},\P}_t  \coloneqq  \exp\bigg(-\int_0^t k_s^{\smallertext{\rm P}} \big(\beta_s^\P \big) \drm s \bigg), \; t \in [0,T],
	\]
	where $k^{\smallertext{\rm P}} : [0,T] \times \Omega \times B \longrightarrow \R$ is assumed to be Borel-measurable, bounded, and such that for any $b \in B$, the map $[0,T]\times\Omega\ni(t,x)\longmapsto k^{\smallertext{\rm P}}_t(x,b)\in\R$ is $\F$--progressively measurable.
\end{enumerate}
Note that, as is customary in principal--agent problems, if there is more than one optimal response for the agent to a given contract, we assume that the principal selects the one that is most favourable to her. This is consistent with the original economic setting considered historically in principal--agent problems, in which the principal suggests an effort when offering the contract, and the agent simply verifies that this effort is indeed optimal for him. Moreover, since by convention the supremum over an empty set is taken to be minus infinity, the principal should always choose a contract such that a best response exists (unless no such contract exists, which leads to a very degenerate problem, and is thus a situation we ignore). 

\begin{remark}
A straightforward generalisation of the above model would allow

\medskip
	$(i)$ the agent’s utility function $U$ to depend on the path of the output process $X_{\cdot\wedge \smallertext{T}}$. No further technical difficulty arises in this case, only invertibility with respect to the $\xi$-variable should be required$;$
	
	\medskip
$(ii)$ the principal to have her own control process $\alpha \in \Ac$, namely an $\F$-predictable $A$-valued process,  for some subset $A$ of a Polish space. This control may represent, for example, inter-temporal payments to the agent, or efforts that the principal herself can exert on the output process $X$, and can appear in all the above defined maps. The contract should then be understood as the pair $(\alpha,\xi)$, meaning that the principal also announces her control at the beginning of the contracting period, see for example {\rm \citeayn{hernandez2026closed}}. This modification mainly leads to more cumbersome notations, but no additional technical difficulty.
\end{remark}

\subsection{Previous results}\label{ss:results}

As already mentioned in the introduction, there exist two approaches to solve principal--agent problems with volatility control. In this section, we briefly describe these two methods, which will later be used to introduce a new form of contracts and verify its optimality.  For both approaches in \cite{cvitanic2018dynamic,chiusolo2026new}, we first define the sets of attainable squared diffusion coefficients and corresponding constrained controls, for $(t,x) \in [0,T] \times \Omega$
\begin{align*}
	{\bf \Sigma}_t(x) \coloneqq  \big\{ [\sigma \sigma^\top]_t(x,b) \in \Sc^{\smallertext{+}}_d(\R) : b \in B \big\}, \text{ and } B_t(x,S) \coloneqq  \big\{ b \in B : [\sigma \sigma^\top]_t(x,b) = S \big\},
	\; S\in{\bf \Sigma}_t(x).
\end{align*}
We then introduce, for any $(t,x,y,z,\gamma) \in [0,T] \times \Omega \times \R \times \R^d \times \Sc_d(\R)$, the agent's Hamiltonian
\begin{align}\label{def_hamiltonian_CPT}
	H_t(x, y, z, \gamma) \coloneqq  \sup_{b \in \smallertext{B}}  h_t (x, y, z, \gamma,b), \;
	\text{with}\;
	h_t (x, y, z, \gamma,b) \coloneqq  f_t (x,y,z,b) + \dfrac12 {\rm Tr} \big[ \gamma [ \sigma \sigma^\top]_t(x,b) \big], \; b \in B,
\end{align} 
as well as its `constrained' version to a specific attainable diffusion $S\in{\bf \Sigma}_t(x)$
\begin{align}\label{eq:hamiltonian_constrained}
	F_t(x,y,z,S) &\coloneqq  \sup_{b \in \smallertext{B}_\smalltext{t}(x,\smallertext{S})} f_t (x,y,z,b),
	\; \mbox{with}\; f_t (x,y,z,b) \coloneqq  \mu_t(x, b) \cdot z - c_t(x,b) - k_t(x, b) y,
	\; b \in B.
\end{align}
We note the following relationship between $H$ and $F$, namely that $2H$ is the Fenchel conjugate of $-2F$,
\begin{align}\label{eq:link_hamiltonian}
	H_t(x,y,z,\gamma) = \sup_{\smallertext{S} \in \smallertext{{\bf \Sigma}}_\smalltext{t}(x)} \bigg\{ F_t(x,y,z,S) + \dfrac12 {\rm Tr} [ S \gamma] \bigg\}.
\end{align}

\subsubsection{The original approach}\label{sss:2BSDE}

Following the approach in \cite{cvitanic2018dynamic}, we are led to consider terminal payments of the form $\xi = U^{(\smalltext{-}1)} \big(Y_\smallertext{T}^{y_\smalltext{0},\smallertext{Z},\smallertext{\Gamma}}\big)$, where the process $Y^{y_\smalltext{0},\smallertext{Z},\smallertext{\Gamma}}$ is defined as the solution to the following SDE
\begin{align}\label{eq:contract_CPT}
	Y_t^{y_\smalltext{0} ,\smallertext{Z},\smallertext{\Gamma}} = y_0 - \int_0^t H_s \big(Y^{y_\smalltext{0} ,\smallertext{Z},\smallertext{\Gamma}}_s, Z_s, \Gamma_s \big) \drm s + \int_0^t Z_s \cdot \drm X_s + \dfrac12 \int_0^t {\rm Tr} \big[ \Gamma_s \drm [ X ]_s \big], \; \P\textnormal{--a.s.}, \;  \P \in \Pc,
\end{align}
for a constant $y_0 \in \R$ and a pair $(Z,\Gamma)$ of processes satisfying appropriate integrability conditions. More precisely, for any $\F$-predictable, $\R^d$-valued process $Z$ and any $\F$-adapted, c\`adl\`ag, $\R$-valued process $Y$, we consider the norms
\begin{align}\label{eq:norm}
	\| Z \|^p_{\H^\smalltext{p}} \coloneqq  \sup_{\P \in \Pc} \E^\P \Bigg[ \bigg( \int_0^\smallertext{T}  Z_t^\top \widehat \sigma^2_t Z_t \drm t \bigg)^{p/2} \Bigg],
	\;
	\| Y \|^p_{\D^\smalltext{p}} \coloneqq  \sup_{\P \in \Pc} \E^\P \bigg[ \sup_{t \in [0,\smallertext{T}]} \big| Y_t \big|^p \bigg].
\end{align}
We can then define the set $\Vc$ of appropriate processes $(Z,\Gamma)$ as in \cite[Definition 3.2]{cvitanic2018dynamic}, recalled below.

\begin{definition}\label{def:contrat_vol}
	Let $y_0 \in \R$. For any $\F$-predictable processes $(Z,\Gamma)$ taking values in $\R^d \times \Sc_d(\R)$, define the process $Y^{y_\smalltext{0} ,\smallertext{Z},\smallertext{\Gamma}}$ using \eqref{eq:contract_CPT}. We will write $(Z,\Gamma) \in \Vc$ if moreover
	\begin{enumerate}[label=$(\roman*)$]
		\itemsep0em
		\item $\| Z \|_{\H^\smalltext{p}} < \infty$ and $\| Y^{y_\smalltext{0} ,\smallertext{Z},\smallertext{\Gamma}} \|_{\D^\smalltext{p}} < \infty$ for some $p > 1;$
		\item there exists $\P \in \Pc$ such that $H_t \big(Y_t^{y_\smalltext{0} ,\smallertext{Z},\smallertext{\Gamma}}, Z_t, \Gamma_t\big) = h_t \big(Y_t^{y_\smalltext{0},\smallertext{Z},\smallertext{\Gamma}}, Z_t, \Gamma_t,\beta^\P_t \big), \; \drm t \otimes \P\text{\rm --a.e. on } [0,T] \times \Omega$.
	\end{enumerate}
\end{definition}

It is straightforward to show that any $\xi = U^{(\smalltext{-}1)} \big( Y_\smallertext{T}^{y_\smalltext{0} ,\smallertext{Z},\smallertext{\Gamma}}\big)$, with $(y_0,Z,\Gamma) \in \R \times \Vc$, is admissible (see \cite[Proposition 3.3]{cvitanic2018dynamic}), and that the agent's optimal response can be characterised as the maximisers of $H$; but the proof of \cite[Theorem 3.6]{cvitanic2018dynamic} may fail if the following assumption is not satisfied (see \Cref{rk:strong_duality}), as shown in \cite{chiusolo2026new}.

\begin{assumption}\label{ass:duality}
	There exists a Borel-measurable function $\gamma^\star : [0,T] \times \Omega \times \R \times \R^d \times\Sc^{\smallertext{+}}_d(\R) \longrightarrow \Sc_d(\R)$ such that for all $(t,\vartheta) \in [0,T] \times \Omega \times \R \times \R^d$ and $S \in {\bf \Sigma}_t(x)$, we have $F_t (\vartheta,S) = H_t  (\vartheta, \gamma_t^\star(\vartheta,S) ) - \frac12 {\rm Tr} [ \gamma_t^\star(\vartheta,S) S ].$
\end{assumption}

\subsubsection{The alternative approach}\label{sss:BSDE}

The approach recently developed in \cite{chiusolo2026new} consists in letting the principal directly controls the density of the quadratic variation, in addition to choosing a terminal payment. More precisely, let
\begin{align}\label{eq:set_qvar}
	\Sc \coloneqq  \big\{\Sigma:[0,T]\times\Omega\longrightarrow \Sc^{\smallertext{+}}_d(\R) :\Sigma \;\text{is}\; \F\text{--progressively measurable}\big\}.
\end{align}
For fixed $\Sigma \in \Sc$, the agent's admissible efforts are now limited to those that achieve the desired quadratic variation. 
\begin{align*}
	\Bc^\circ(\Sigma,\P) &\coloneqq  \big\{\beta \in \Bc  :  \beta_t \in B_t(\Sigma_t),\; \mathrm{d}t\otimes\P\text{\rm--a.e.}\big\}, \; \text{and} \; 
	\Pc^\circ(\Sigma) \coloneqq  \big\{ \P \in \Pc  :  [\sigma \sigma^\top ]_t (\beta_t^\P) = \Sigma_t, \; \mathrm{d}t\otimes\P\text{--a.e.}\big\}.
\end{align*}
As already mentioned, for this approach we actually need to restrict the set of admissible controls to those inducing a unique weak solution to the original SDE \eqref{eq:dynamic_SDE}. We refer to \cite[Assumption 2.2]{chiusolo2026new} and related remarks for a rigorous statement and further justifications. Given $\xi \in \L^0(\Fc_{\smallertext{T}},\R)$ and $\Sigma \in \Sc$, the optimisation problem faced by the agent is
\begin{align*}
	V^\circ_{\smallertext{\rm A}} (\xi,\Sigma) \coloneqq  \sup_{\P \in \Pc^\smalltext{\circ}(\smallertext{\Sigma})} J_{\smallertext{\rm A}} (\xi,\P).
\end{align*}
In parallel to \eqref{def:contract_set} and \eqref{def:optimal_response}, 
we define the set $\Xi^\circ$ of admissible contracts as
\begin{align*}
	\Xi^\circ \coloneqq  \big\{(\xi, \Sigma) \in\L^0(\Fc_\smallertext{T},\R) \times \Sc :\text{$\xi$ satisfies \eqref{eq:integrability_contract_agent} and} \; V^\circ_{\smallertext{\rm A}} (\xi,\Sigma) \geq R_{\smallertext{\rm A}} \big\},
\end{align*}
and denote by $\Pc^{\circ,\star}(\xi,\Sigma)$ the set of agent's optimal response to a contract, here a pair $(\xi, \Sigma) \in \Xi^\circ$.
Then, the goal of the principal is to maximise her objective function by choosing $(\xi, \Sigma) \in \Xi^\circ$, anticipating the agent's optimal response
\begin{align}\label{eq:pb_principal_FB}
	V_{\smallertext{\rm P}}^\circ \coloneqq  \sup_{(\xi, \smallertext{\Sigma}) \in \smallertext{\Xi}^\circ}  \sup_{\P \in \Pc^{\smalltext{\circ}\smalltext{,}\smalltext{\star}}(\xi,\smallertext{\Sigma})} J_{\smallertext{\rm P}} (\xi,\P).
\end{align}

This auxiliary problem is not meant to be economically realistic; rather, it is introduced as a tractable benchmark: its value is easier to characterise than that of the original problem, and we then show that the two values coincide.
More precisely, it is first proved in \cite[Lemma 3.5]{chiusolo2026new} that $V_{\smallertext{\rm P}}^\circ \geq V_{\smallertext{\rm P}}$, namely $V_{\smallertext{\rm P}}^\circ$ defined above is the best possible value for the principal, if achievable. Then, the above problem can be solved using BSDEs. Indeed, for this alternative problem, \cite[Theorem 3.10]{chiusolo2026new} states that considering terminal payments of the form $\xi = U^{(\smalltext{-}1)} (Y_\smallertext{T}^{y_\smalltext{0} ,\smallertext{Z},\circ})$, with 
\begin{align}\label{eq:contract_FB}
	Y_t^{y_\smalltext{0} ,\smallertext{Z},\circ} = y_0 - \int_0^t F_s\big(Y^{y_\smalltext{0} ,\smallertext{Z},\circ}_s, Z_s,\Sigma_s \big) \drm s + \int_0^t Z_s \cdot \drm X_s, \; \P\textnormal{--a.s.}, \; \text{for all} \; \P \in \Pc^\circ (\Sigma),
\end{align}
is without loss of generality from the principal's point of view. 
\begin{definition}\label{def:contract_FB}
	Let $y_0 \in \R$. For any $\Sigma \in \Sc$ and any $\F$-predictable processes $Z$ taking values in $\R^d$, define the process $Y^{y_\smalltext{0} ,\smallertext{Z},\circ}$ using \eqref{eq:contract_FB}. We will denote $(Z,\Sigma) \in \Vc^\circ$ if moreover
	\begin{enumerate}[label=$(\roman*)$]
		\item $\| Z \|^p_{\H^\smalltext{p}} < \infty$ and $\| Y^{y_\smalltext{0} ,\smallertext{Z},\circ} \|^p_{\D^\smalltext{p}} < \infty$ for some $p > 1;$
		\item \label{item2_def_contract} there exists $\P \in \Pc^\circ(\Sigma)$ such that $F_t \big(Y_t^{y_\smalltext{0} ,\smallertext{Z},\circ}, Z_t, \Sigma_t \big) = f_t \big(Y_t^{y_\smalltext{0} ,\smallertext{Z},\circ}, Z_t, \beta_t^{\P} \big)$, $\drm t \otimes \P$--{\rm a.e.} on $[0,T] \times \Omega$.
	\end{enumerate}
\end{definition}
Then, \cite[Proposition 3.9]{chiusolo2026new} establishes that the agent's optimal response to such contract $(\xi,\Sigma)$ is given by a maximiser of the constrained Hamiltonian. More precisely, the agent's optimal effort can be represented by $\mathfrak{b}^\circ_\cdot \coloneqq  b^\circ_\cdot( Y_\cdot^{y_\smalltext{0},\smallertext{Z},\circ}, Z_\cdot,\Sigma_\cdot)$, where the function $b^\circ$ is defined by
\begin{align}\label{eq:u*_FB}
	b^\circ_t (x,y,z,S) \in \argmax_{b \in \smallertext{B}_\smalltext{t}(x,\smallertext{S})} f_t ( x, y, z,b), \; (t,x,y,z,S) \in [0,T] \times \Omega \times \R \times \R^d \times \Sc^{\smallertext{+}}_{d}(\R).
\end{align} 
We further define $\Pc^{\circ,\star} (Z,\Sigma) \coloneqq \Pc^{\circ,\star} \big( U^{(\smalltext{-}1)} \big(Y_\smallertext{T}^{y_\smalltext{0},\smallertext{Z},\circ} \big), \Sigma \big).$ To summarise, we obtained
\begin{align}\label{eq:FB_solution}
	V_{\smallertext{\rm P}}^\circ = \sup_{y_\smalltext{0} \geq \smallertext{R}_{\smalltext{\rm A}}}  \sup_{(\smallertext{Z}, \smallertext{\Sigma}) \in \Vc^\circ}  \sup_{\P \in \Pc^{\smalltext{\circ}\smalltext{,}\smalltext{\star}}(\smallertext{Z},\smallertext{\Sigma})} J_{\smallertext{\rm P}} \Big( U^{(\smalltext{-}1)} \big(Y_\smallertext{T}^{y_\smalltext{0} ,\smallertext{Z},\circ} \big),\P \Big).
\end{align}
More importantly, \cite[Theorem 3.12]{chiusolo2026new} proves that this value can be achieved using contracts of the form \eqref{eq:contract_CPT}, under \Cref{ass:duality}. Subsequently, this provides an alternative BSDE-based approach to prove optimality of the contract form \eqref{eq:contract_CPT}. However, since the two aforementioned approaches may fail when \Cref{ass:duality} is not satisfied, it is crucial to investigate the existence of optimal contracts beyond this setting. This is the goal of next section.

\section{General optimal contracts}\label{sec:contract_general}

We first use the 2BSDE characterisation of the agent's continuation utility to introduce a general form of contracts. We then derive appropriate conditions ensuring that the contract form is \emph{revealing} for the agent and \emph{without loss of generality} for the principal. These results are given in \Cref{prop:agent} and \Cref{thm:principal}. The proof of \Cref{prop:agent} follows classical arguments, but is included here for the sake of completeness. The proof of \Cref{thm:principal} relies on both approaches briefly described in the previous sections, and is detailed in \Cref{sss:2BSDE_proof,sss:BSDE_proof} respectively.

\subsection{Deriving a general form of contracts}\label{ss:general_contracts}

As established in \cite[Proposition 4.6]{cvitanic2018dynamic}, for a fixed contract $\xi \in \Xi$ in the original problem, the agent's continuation utility $Y$ is the first component of the solution $(Y,Z,K)$ to the 2BSDE 
\begin{align}\label{eq:agent_2BSDE}
	Y_t = {U} ( \xi) + \int_t^T F_s\big(Y_s, Z_s,\widehat \sigma^2_s \big) \drm s - \int_t^T Z_s \cdot \drm X_s + \int_t^T \drm K_s, \; \P\textnormal{--a.s.}, \; \text{for all} \; \P \in \Pc,
\end{align}
where $K$ is a non-decreasing process starting from $K_0=0$, see \cite[Definition 4.4]{cvitanic2018dynamic}. 
Adapting the arguments in the proof of Theorem 3.6 in \cite[Section 4.6]{cvitanic2018dynamic}, as detailed later in \Cref{sss:2BSDE_proof} for completeness, we can assume without loss of generality that the non-decreasing process $K$ in the 2BSDE is absolutely continuous with respect to the Lebesgue measure, with non-negative density denoted $\dot K$. This leads to $\xi = U^{(\smalltext{-}1)}(Y_\smallertext{T})$ for the terminal payment, with
\begin{align}\label{eq:xi_2BSDE}
	Y_t = y_0 - \int_0^t \big( F_s \big(Y_s, Z_s, \widehat \sigma^2_s \big) + \dot K_s \big) \drm s + \int_0^t Z_s \cdot \drm X_s,\; t\in[0,T], \; \P\textnormal{--a.s.}, \; \text{for all} \; \P \in \Pc.
\end{align}
From there, the idea is to find a suitable representation for the density $\dot K$, leading to a family of parametric contracts $(\xi^\theta)_{\theta \in \Theta}$, or an appropriate parameter set $\Theta$ such that each contract $\xi^\theta$ is
\begin{itemize}
	\item[$(\Rc)$]\hypertarget{prop:revealing}{} \emph{revealing}, in the sense that, for every $\theta \in \Theta$, $\xi^\theta$ is admissible, $V_{\smallertext{\rm A}}(\xi^\theta) = y_0 \ge R_\smallertext{A}$, and there exists a process $\mathfrak{b}^{\star,\theta} \in \Bc$, depending on $\theta$ and the observable states, such that the optimal responses of the agent are exactly the $\P^\star \in \Pc$ satisfying $\beta^{\P^\star}_t = \mathfrak{b}^{\star,\theta}_t$, $\drm t \otimes \P^\star$--a.e.;
	\item[$(\Wc)$]\hypertarget{prop:wlog}{} \emph{without loss of generality} from the principal's point of view, in the sense that restricting her optimisation to the sub-family $\{\xi^\theta\}_{\theta \in \Theta} \subseteq \Xi$ leaves her value unchanged, namely
	\begin{align*}
		V_{\smallertext{\rm P}} = \sup_{\theta \in \smallertext{\Theta}} \sup_{\P^\smalltext{\star} \in \Pc^\smalltext{\star}(\xi^\smalltext{\theta})} J_{\smallertext{\rm P}}(\xi^\theta, \P^\star).
	\end{align*}
\end{itemize}
We will see below that we can take the parameter of the form $\theta = (y_0, Z, \Gamma, \psi)$, where $\psi$ belongs to the class $\Psi$ of \Cref{def:relevant_psi}: it satisfies \Cref{cond:revealing}, ensuring property~\hyperlink{prop:revealing}{$(\Rc)$} (see \Cref{prop:agent}), together with \Cref{cond:wlog1} and either \ref{thm:cond_a} or \ref{thm:cond_b} of that definition, ensuring property~\hyperlink{prop:wlog}{$(\Wc)$} (see \Cref{thm:principal}). 

\medskip

First, in order for the contract to be implementable, the process $\dot K$ can only depend on what is observable by the principal. We thus assume
\begin{align}\label{eq:representation_Kdot}
	\dot K_t = \psi_t \big(Y_t,Z_t,\Gamma_t,\widehat \sigma^2_t\big) - F_t \big(Y_t, Z_t, \widehat \sigma^2_t \big), \; t \in [0,T],
\end{align}
for some function $\psi : [0,T] \times \Omega \times \R \times \R^d \times E \times \Sc^{\smallertext{+}}_d(\R) \longrightarrow \R$, where $E$ is an arbitrary space. Similarly to \eqref{eq:contract_CPT} and \Cref{def:contrat_vol}, we consider the process $Y^{y_\smalltext{0},\smallertext{Z},\smallertext{\Gamma},\psi}$, defined as a solution to the following SDE for $t\in [0,T]$
\begin{align}\label{eq:xi_2BSDE_psi}
\begin{split}
	Y^{y_\smalltext{0},Z,\Gamma,\psi}_t &= y_0 - \int_0^t \psi_s \big(Y^{y_\smalltext{0},\smallertext{Z},\smallertext{\Gamma},\psi}_s,Z_s,\Gamma_s,\widehat \sigma^2_s\big) \drm s 
	+ \int_0^t Z_s \cdot \drm X_s, \; \P\textnormal{--a.s.},
\end{split}
\end{align}
for a constant $y_0 \in \R$ and a pair $(Z,\Gamma) \in \Vc^\psi$ defined below. To ensure that the above SDE has a unique strong solution, we will further require that the function $\psi$ is Lipschitz-continuous in the $y$-variable, uniformly in the others, and denote by $\Oc $ the set of such Borel-measurable functions $\psi : [0,T] \times \Omega \times \R \times \R^d \times E \times \Sc^{\smallertext{+}}_d(\R) \longrightarrow \R$.

\begin{definition}\label{def:contrat_new}
	Let $\psi \in \Oc $ and $y_0 \in \R$. For any $\F$-predictable processes $(Z,\Gamma)$ taking values in $\R^d \times E$, define the process $Y^{y_\smalltext{0},\smallertext{Z},\smallertext{\Gamma},\psi}$ using \eqref{eq:xi_2BSDE_psi}. We will write $(Z,\Gamma) \in \Vc^\psi$ if moreover
	\begin{enumerate}[label=$(\roman*)$]
		\itemsep0em
		\item for some $p > 1$, $\| Z \|_{\H^\smalltext{p}} < \infty$, $\| Y^{y_\smalltext{0},\smallertext{Z},\smallertext{\Gamma},\psi} \|_{\D^{\smalltext{p}}} < \infty$ and
		\begin{align*}
			\sup_{\P \in \Pc} \E^\P \Bigg[ \bigg( \int_0^\smallertext{T}   \psi_s \big(0,Z_s,\Gamma_s,\widehat \sigma^2_s \big) \drm s \bigg)^{p} \Bigg] < + \infty;
		\end{align*}
		\item there exists $\P \in \Pc$ such that
		\begin{align}\label{eq:def-optimal}
			\psi_t \big(Y^{y_\smalltext{0},\smallertext{Z},\smallertext{\Gamma},\psi}_t,Z_t,\Gamma_t,\widehat \sigma^2_t\big) = F_t \big(Y_t^{y_\smalltext{0}, \smallertext{Z},\smallertext{\Gamma},\psi}, Z_t,\widehat \sigma^2_t \big) = f_t \big(Y_t^{y_\smalltext{0}, \smallertext{Z},\smallertext{\Gamma},\psi}, Z_t, \beta^\P_t \big), \; \drm t \otimes \P\text{\rm --a.e. on } [0,T] \times \Omega.
		\end{align}
	\end{enumerate}
\end{definition}

Overall, we will consider the following class of contracts.
\begin{definition}[Relevant class of contracts]\label{def:relevant_contracts}
	A contract $\xi$ belongs to the class of relevant contracts, denoted $\Cc$, if there exist a function $\psi \in \Oc $, a constant $y_0 \in \R$ and a pair of processes $(Z,\Gamma) \in \Vc^\psi$ such that $\xi = U^{(\smalltext{-}1)} (Y_\smallertext{T}^{y_{\smalltext{0}},\smallertext{Z},\smallertext{\Gamma},\psi})$, for $Y^{y_{\smalltext{0}}, \smallertext{Z},\smallertext{\Gamma},\psi}$ defined as in \eqref{eq:xi_2BSDE_psi}. We denote by $\Pc_\psi^\star (Z,\Gamma)$ the set of agent's best response to such contract.
\end{definition}

\subsection{Main results}\label{ss:main_results}

Given a relevant contract $\xi \in \Cc$ as defined above, thus parametrised by a function $\psi \in \Oc $, a constant $y \geq R_{\smallertext{\rm A}}$, and a pair $(Z,\Gamma) \in \Vc^\psi$, one can show that if the function $\psi$ satisfies in addition \Cref{cond:revealing} below, then the contract $\xi$ is admissible and one can easily compute the agent's optimal response, so that the revealing property \hyperlink{prop:revealing}{$(\Rc)$} is satisfied. 
\begin{proposition}[Agent's optimal response]\label{prop:agent}
	Let $\psi \in \Oc $, $y_0 \geq R_{\smallertext{\rm A}}$ and $(Z,\Gamma) \in \Vc^\psi$, and consider the associated contract $\xi \in \Cc$, in the sense of {\rm \Cref{def:relevant_contracts}}. If the map $\psi$ satisfies
	\begin{align}\label{cond:revealing}
		\inf_{\smallertext{S} \in \smallertext{{\bf \Sigma}}_\smalltext{t}(x)} \big\{ \psi_t(x,y,z,\gamma,S) - F_t(x,y,z,S) \big\} = 0, 
		\; \text{\rm for all} \; (t,x,y,z,\gamma) \in [0,T] \times \Omega \times \R \times \R^d \times E,
		\tag{${\rm C}_{\smallertext{\Rc}}$}
	\end{align}
	then $\xi \in \Xi$, $V_{\smallertext{\rm A}} (\xi) = y_0 \geq R_{\smallertext{\rm A}}$, and all agent's optimal responses take the form $\mathfrak{b}^{\star,\psi}_\cdot \coloneqq b^{\star,\psi}_\cdot(Y^{y_\smalltext{0},\smallertext{Z},\smallertext{\Gamma},\psi}_\cdot,Z_\cdot,\Gamma_\cdot)$, where $b^{\star,\psi}$ satisfies for all $(t,x,y,z,\gamma) \in [0,T] \times \Omega \times \R \times \R^d \times E$
	\begin{align}\label{eq:optimum}
		b^{\star,\psi}_t(x,y,z,\gamma) \in \argmax_{b \in \smallertext{B}_\smalltext{t}(x,\smallertext{S}^{\smalltext{\star}\smalltext{,}\smalltext{\psi}}_\smalltext{t}(x,y,z,\gamma))}  f_t (x,y,z,b), \; S^{\star,\psi}_t(x,y,z,\gamma) \in \argmin_{\smallertext{S} \in \smallertext{{\bf \Sigma}}_\smalltext{t}(x)} \big\{ \psi_t(x,y,z,\gamma,S) - F_t(x,y,z,S) \big\} .
	\end{align}
	For later reference, we call $\mathfrak B^{\star,\psi}$ the set of all such functions $b^{\star,\psi}$.
\end{proposition}

The previous result echoes \cite[Proposition 3.3]{cvitanic2018dynamic} for the contract form \eqref{eq:contract_CPT} and \cite[Proposition 3.9]{chiusolo2026new} for  \eqref{eq:contract_FB}. The proof follows standard arguments and is therefore postponed to \Cref{sec:app}.

\begin{remark}\label{rk1:psi=0}
	We note that if the map $\psi$ is taken to be identically equal to $F$, the maximisation \eqref{eq:optimal_effort} boils down to the maximisation over $S \in {\bf \Sigma}_{t}(x)$ of the following quantity $\sup_{b \in \smallertext{B}_\smalltext{t}(x,\smallertext{S})} \big\{ f_t(x,y,z,b) \big\} - F_t(x,y,z,S),$ which is always $0$ by definition of $F$. In other words, when taking $\psi \equiv F$ in the contract, the agent is actually indifferent between any achievable quadratic variations. In the end, we will note in {\rm \Cref{rk2:psi=0}} that since the principal can choose among the agent's optimal responses, she will be able to select any desired achievable quadratic variation.
\end{remark}

The next task is to ensure that the contract form introduced through \eqref{eq:xi_2BSDE_psi} satisfies Property \hyperlink{prop:wlog}{$(\Wc)$}. More precisely, for a fixed parametrisation function $\psi \in \Oc $ satisfying \Cref{cond:revealing}, the principal's value function is given by
\begin{align}\label{eq:pb_principal_standard}
	\overline V^\psi_{\smallertext{\rm P}} \coloneqq  \sup_{y_\smalltext{0} \geq R_{\smalltext{\rm A}}} \widetilde V^\psi_{\smallertext{\rm P}} (y_0), \; \text{\rm with} \; 
	\widetilde V^\psi_{\smallertext{\rm P}} (y_0) \coloneqq  \sup_{(\smallertext{Z},\smallertext{\Gamma}) \in \smallertext{\Vc}^\smalltext{\psi}} \sup_{\P \in \Pc^\smalltext{\star}_\smalltext{\psi} (\smallertext{Z},\smallertext{\Gamma})} J_{\smallertext{\rm P}} \Big( U^{(\smalltext{-}1)}\big (Y_\smallertext{T}^{y_{\smalltext{0}},\smallertext{Z},\smallertext{\Gamma},\psi}\big), \P \Big).
\end{align}
As proved in \Cref{prop:agent}, any contract $\xi \in \Cc$ belongs to the original set of admissible contracts $\Xi$, thus ensuring that for all $\psi \in \Oc$, the value $\overline V^\psi_{\smallertext{\rm P}}$ defined above is lower than the value $V_{\smallertext{\rm P}}$. Moreover, under weak uniqueness, it has been noticed in \cite[Lemma 3.5]{chiusolo2026new} that the latter is smaller than the \textit{contractible-volatility} value, hence $V^\circ_{\smallertext{\rm P}} \geq V_{\smallertext{\rm P}} \geq \overline V^\psi_{\smallertext{\rm P}}$. While \cite{chiusolo2026new} proves that the equalities hold under \Cref{ass:duality}, using the contract form originating from \cite{cvitanic2018dynamic}, one can actually prove that this result extends beyond this assumption, using two complementary approaches.

\medskip
The first approach leverages the 2BSDE-based narrative developed in \cite{cvitanic2018dynamic} and briefly recalled in \Cref{sss:2BSDE}. In particular, one needs there to ensure that any non-negative density process $\dot K$ in \eqref{eq:xi_2BSDE} can be represented using a function $\psi$ as described in \eqref{eq:representation_Kdot}. In other words, the map $E\ni \gamma  \longmapsto \psi_t (x,y,z,\gamma,S) - F_t (x,y,z,S) \in \R$ needs to be surjective onto $[0,+\infty)$. It is sufficient to assume the above map is {continuous} and that $\psi$ satisfies
\begin{subequations}\label{cond:wlog}
	\begin{gather}
		\min_{\gamma \in \smallertext{E}} \psi_t(x,y,z,\gamma,S) = F_t(x,y,z,S), \tag{${\rm C}^{\rm min}_{\smallertext{\Wc}}$}
		\label{cond:wlog1}\\
		\sup_{\gamma \in \smallertext{E}} \psi_t(x,y,z,\gamma,S) = +\infty, \tag{${\rm C}^{\rm sup}_{\smallertext{\Wc}}$}
		\label{cond:wlog2}
	\end{gather}
\end{subequations}
for all $(t,x,y,z) \in [0,T] \times \Omega \times \R \times \R^d$ and $S \in {\bf \Sigma}_t(x)$.
With the above requirements on $\psi$, one can follow the 2BSDE approach to show that the value $\overline V^\psi_{\smallertext{\rm P}}$ defined in \eqref{eq:pb_principal_standard} achieves the original value $V_{\smallertext{\rm P}}$.

\medskip

Alternatively, one can use the BSDE-based approach, proposed in \cite{chiusolo2026new} and summarised in \Cref{sss:BSDE}. As already mentioned, this approach has the drawback that it requires to assume weak uniqueness for the SDE \eqref{eq:dynamic_SDE}. Nonetheless, one can show that if the parametrisation function $\psi \in \Oc $ satisfies {\rm\Cref{cond:revealing}}, \Cref{cond:wlog1} and some additional regularity conditions detailed in \Cref{ass:psi_star}, the associated contract can be proved to be optimal. Although \Cref{ass:psi_star} may be cumbersome to verify for an arbitrary $\psi \in \Oc $, it suffices to note that $\psi \equiv F$ works.

\begin{condition}\label{ass:psi_star}
	Let $\psi \in  \Oc $ satisfying {\rm\Cref{cond:wlog1}}, and define for $(t,x,y,z) \in [0,T] \times \Omega \times \R \times \R^d$,  $S \in {\bf \Sigma}_t(x)$
	\begin{align}\label{eq:gamma_star}
		\gamma^\star_t(x,y,z,S)  \in \argmin_{\gamma \in \smallertext{E}} \psi_t \big(x,y,z,\gamma,S\big).
	\end{align}
	The map $\psi \circ \gamma^\star : (t,x,y,z,S,S^\prime) \longmapsto \psi_t(x,y,z,\gamma_t^\star (x,y,z,S),S^\prime)$ is Lipschitz-continuous in $y$ and has linear growth in $z$, namely there exist $L> 0$ and $C > 0$ such that for all $(t,x,y,y^\prime,z) \in [0,T] \times \Omega \times \R \times \R \times \R^d$ and $(S, S') \in {\bf \Sigma}_t(x)^2$
	\begin{gather*}
		\big| \psi_t \circ \gamma^\star_t (x,y,z,S,S^\prime) - \psi_t \circ \gamma_t^\star (x,y^\prime,z,S,S^\prime) \big| 
		\leq L \left| y - y' \right|, \; \text{\rm and} \; 
		\left| \psi_t \circ \gamma_t^\star (x,0,z,S,S^\prime) \right| 
		\leq C \big(1 + \|x\| + \| z^\top S' z \|^{1/2} \big).
	\end{gather*}
\end{condition}

\begin{definition}[Relevant parametrisation functions]\label{def:relevant_psi}
	We let $\Psi$ be the subset of continuous functions $\psi \in \Oc $ satisfying {\rm\Cref{cond:revealing,cond:wlog1}}, and such that either
	\begin{enumerate}[label=$(\alph*)$]
		\itemsep0em
		\item \label{thm:cond_a} {\rm\Cref{cond:wlog2}} is satisfied$;$
		\item \label{thm:cond_b} or {\rm \Cref{ass:psi_star}} holds.
	\end{enumerate}
\end{definition}
The following theorem states the main results regarding the optimal form of contracts and the equality of the values.

\begin{theorem}\label{thm:principal}
	Let $\psi \in \Psi$, then the contract form parametrised by $\psi$ is optimal. More precisely
	\begin{enumerate}[label=$(\roman*)$]
		\item if $\psi$ satisfies \ref{thm:cond_a}, $\overline V^\psi_{\smallertext{\rm P}} = V_{\smallertext{\rm P}}$;
		\item if $\psi$ satisfies \ref{thm:cond_b} and weak uniqueness is assumed for {\rm SDE} \eqref{eq:dynamic_SDE}, $\overline V^\psi_{\smallertext{\rm P}} = V^\circ_{\smallertext{\rm P}} = V_{\smallertext{\rm P}}$.
	\end{enumerate}
\end{theorem}

Before turning to the proofs, some remarks are in order. First, for fixed $y_0 \in \R$, the value $\widetilde V^\psi_{\smallertext{\rm P}}(y_0)$ now corresponds to the value of a more standard stochastic control problem. It admits two state variables, $X$ and $Y^{y_{\smalltext{0}},\smallertext{Z},\smallertext{\Gamma},\psi}$, whose dynamics are respectively given by \eqref{eq:dynamic_SDE} and \eqref{eq:xi_2BSDE_psi}, indirectly controlled by the principal through the agent's best response to a terminal payment $\xi = U^{(\smalltext{-}1)}\big (Y_\smallertext{T}^{y_{\smalltext{0}},\smallertext{Z},\smallertext{\Gamma},\psi}\big)$. 
More precisely, for an optimal response $\mathfrak{b}^{\star,\psi}_t \coloneqq  b^{\star,\psi}_t(Y^{y_{\smalltext{0}},\smallertext{Z},\smallertext{\Gamma},\psi}_t,Z_t,\Gamma_t)$, $\mathrm{d}t \otimes \P^\star$--a.e., characterised via \eqref{eq:optimum} and a corresponding probability measure $\P^\star \in \Pc^\star_\psi (Z,\Gamma)$, the pair $(X,Y^{y_{\smalltext{0}},\smallertext{Z},\smallertext{\Gamma},\psi})$ is a solution to the following system of controlled SDEs
\begin{gather*}
		\drm X_t  = \mu_t \big(b^{\star,\psi}_t(Y^{y_{\smalltext{0}},\smallertext{Z},\smallertext{\Gamma},\psi}_t,Z_t,\Gamma_t) \big) \drm t + \sigma_t  \big(b_t^{\star,\psi}(Y^{y_{\smalltext{0}},\smallertext{Z},\smallertext{\Gamma},\psi}_t,Z_t,\Gamma_t) \big) \drm W^{\P^\smalltext{\star}}_t, \\[0.5em]
		\drm Y_t^{y_{\smalltext{0}},\smallertext{Z},\smallertext{\Gamma},\psi} = \big( c_t + Y_t^{y_{\smalltext{0}},\smallertext{Z},\smallertext{\Gamma},\psi} k_t \big) \big(b^{\star,\psi}_t(Y^{y_{\smalltext{0}},\smallertext{Z},\smallertext{\Gamma},\psi}_t,Z_t,\Gamma_t) \big) \drm t + Z_t^\top \sigma_t  \big(b^{\star,\psi}_t(Y^{y_{\smalltext{0}},\smallertext{Z},\smallertext{\Gamma},\psi}_t,Z_t,\Gamma_t) \big) \drm W^{\P^\smalltext{\star}}_t,
\end{gather*}
for $t \in [0,T]$, with initial condition $X_0 = x_0$ and $Y^{y_{\smalltext{0}},\smallertext{Z},\smallertext{\Gamma},\psi}_0 = y_0$. Indeed, by computing the dynamics \eqref{eq:xi_2BSDE_psi} under the agent's optimal effort $\mathfrak{b}^\star$, one can notice that by \Cref{cond:revealing} and the characterisation \eqref{eq:optimum}, we have
\begin{align*}
	\psi_t \big(Y^{y_{\smalltext{0}},\smallertext{Z},\smallertext{\Gamma},\psi}_t,Z_t,\Gamma_t,S^{\star,\psi}_t(Y^{y_{\smalltext{0}},\smallertext{Z},\smallertext{\Gamma},\psi}_t,Z_t,\Gamma_t) \big) 
	&= F_t \big(Y^{y_{\smalltext{0}},\smallertext{Z},\smallertext{\Gamma},\psi}_t, Z_t, S_t^{\star,\psi}(Y^{y_{\smalltext{0}},\smallertext{Z},\smallertext{\Gamma},\psi}_t,Z_t,\Gamma_t) \big) \\
	&= f_t \big(Y^{y_{\smalltext{0}},\smallertext{Z},\smallertext{\Gamma},\psi}_t, Z_t, b_t^{\star,\psi}(Y^{y_{\smalltext{0}},\smallertext{Z},\smallertext{\Gamma},\psi}_t,Z_t,\Gamma_t) \big), \; \mathrm{d}t\otimes \P^\star\text{\rm--a.e.}
\end{align*}

One should notice that the function $\psi$ does not explicitly appear in the above SDE system. Indeed, the principal's value is computed under the agent's optimal effort, so any chosen parametrisation function $\psi$ vanishes to $F$, according to \Cref{cond:revealing}. However, the optimal effort characterised through the function $b^{\star,\psi}$ does depend on the choice of $\psi$. Nevertheless, \Cref{thm:principal} states that the principal's value $V_{\smallertext{\rm P}}$ in the original problem, which is by definition independent of the choice of $\psi$, is equal to $\overline V^\psi_{\smallertext{\rm P}}$, implying that such value is also be independent of $\psi$. This will appear more clearly in the proof of the theorem detailed below, but intuitively, one can switch from $\psi \in \Psi$ to another $\phi \in \Psi$ using the fact that they should both satisfy \Cref{cond:revealing,cond:wlog1}.

\medskip

To completely solve the principal's problem, one can then write the Hamilton--Jacobi--Bellman equation satisfied by $V^\psi_{\smallertext{\rm P}} (y_0)$ introduced in \eqref{eq:pb_principal_standard} and look for an appropriate solution. More precisely, the Hamiltonian functional for the principal's problem is now given, for any $(t,x,y,v,p_x,p_y,q_{xx},q_{yy},q_{xy})\in[0,T]\times\Omega\times\R\times\R^d\times\R\times\Sc_d(\R)\times\R\times\R^d$, by
\begin{align*}
	H_t^{\smallertext{\rm P}}(x,y,v,p_x,p_y,q_{xx},q_{yy},q_{xy})\coloneqq \sup_{(z,\gamma)\in\R^\smalltext{d}\times \smallertext{E}}\sup_{b^{\smalltext{\star}\smalltext{,}\smalltext{\psi}}\in\smallertext{\mathfrak{B}}^{\smalltext{\star}\smalltext{,}\smalltext{\psi}}}h_t^{\smallertext{\rm P}}(x,y,v,p_x,p_y,q_{xx},q_{yy},q_{xy},z,\gamma),
\end{align*}
where 
\begin{align*}
	h_t^{\smallertext{\rm P}}(x,y,v,p_x,p_y,q_{xx},q_{yy},q_{xy},z,\gamma)
	\coloneqq &\ \mu_t\big(x,b^{\star,\psi}_t(x,y,z,\gamma)\big)\cdot p_x
	-k^{\smallertext{\rm P}}_t\big(x,b^{\star,\psi}_t(x,y,z,\gamma)\big)v
	-c^{\smallertext{\rm P}}_t\big(x,b_t^{\star,\psi}(x,y,z,\gamma)\big) \\
	&+[ c_t+ y k_t ] \big(x,b_t^{\star,\psi}(x,y,z,\gamma) \big)p_y  +[\sigma\sigma^\top]_t\big(x,b^{\star,\psi}_t(x,y,z,\gamma) \big)z\cdot q_{xy}  \\
	&+\frac12\mathrm{Tr}\big[[\sigma\sigma^\top]_t\big(x,b_t^{\star,\psi}(x,y,z,\gamma)\big)q_{xx}\big]+\frac12\big\|\sigma^\top_t\big(x,b^{\star,\psi}_t(x,y,z,\gamma)\big)\big\|^2q_{yy}.
\end{align*}
So choosing a specific $\psi\in\Psi$ can roughly speaking be related to writing $H^{\smallertext{\rm P}}$ in different manners as a certain supremum, which itself is naturally associated to a dynamic of $Y$ leading to revealing contracts which are without loss of generality for the principal. 

\subsection{Proof of Theorem \ref{thm:principal}}\label{ss:proof_main}

To prove the main theorem, we first claim that $V_{\smallertext{\rm P}} \geq \overline V^\psi_{\smallertext{\rm P}}$ for any $\psi \in \Psi$, which echoes \cite[Proposition 3.4]{cvitanic2018dynamic}. Indeed, let $\xi \in \Cc$ defined by \eqref{eq:xi_2BSDE_psi} with $y_0 \geq R_{\smallertext{\rm A}}$, $\psi \in \Psi$ and $(Z,\Gamma) \in \Vc^\psi$. By \Cref{prop:agent}, $\xi \in \Xi$, \textit{i.e.} $\Cc \subset \Xi$, thus impliying the desired inequality. It remains to establish the converse inequality.

\subsubsection{The 2BSDE road}\label{sss:2BSDE_proof}

In this section, we let $\psi \in \Psi$ such that \ref{thm:cond_a} is satisfied and show that $\overline V^\psi_{\smallertext{\rm P}} \geq V_{\smallertext{\rm P}}$. The main idea of this proof is to show that any admissible contract $\xi \in \Xi$ can be approximated—in an appropriate sense specified below—by some $\xi^\varepsilon \in \Cc$. The reasoning largely follows the arguments detailed in \cite[Section 4]{cvitanic2018dynamic}; we therefore only highlight the main steps while addressing the issue raised in \cite{chiusolo2026new} about \Cref{ass:duality} (see \Cref{rk:strong_duality}).

\medskip

\textit{Step 1.} We first let $\xi \in \Xi$ and argue by \cite[Proposition 4.5]{cvitanic2018dynamic} that the 2BSDE \eqref{eq:agent_2BSDE} has a unique solution $(Y,Z,K)$, satisfying appropriate integrability conditions (see \cite[Definition 4.4]{cvitanic2018dynamic}). On the one hand, by assumption on $\Xi$, there should exist (at least) an optimal response $\P^\star \in \Pc^\star(\xi)$, and thus an associated optimal control $\beta^{\P^\smalltext{\star}} \in \Bc$. For such optimal response, \cite[Proposition 4.6]{cvitanic2018dynamic} states that $K_\smallertext{T} = 0$, $\P^\star$--a.s. and $\beta^{\P^\smalltext{\star}}$ is a maximiser of the `constrained Hamiltonian' defined in \eqref{eq:hamiltonian_constrained}. Then fix $\varepsilon > 0$ and define an absolutely continuous approximation of $K$
\begin{align}\label{eq:approx_K}
	K^\varepsilon_t \coloneqq \dfrac1{\varepsilon} \int_{(t-\varepsilon)^{\smalltext{+}}}^t K_s \drm s, \; t \in [0,T].
\end{align}
We first note that by definition, $0 \leq K^\varepsilon_T \leq K_T$, hence $K^\varepsilon$ inherits the integrability from $K$ and satisfies the same minimality condition. Furthermore, a probability measure $\P$ satisfies $K_\smallertext{T} = 0$, $\P$--a.s. if and only if it satisfies $K^\varepsilon_\smallertext{T} = 0$, $\P$--a.s. We will later denote by $\dot K^\varepsilon$ its (non-negative) density with respect to the Lebesgue measure. 

\medskip

\textit{Step 2.} We then consider the process $Y^\varepsilon$ defined as the unique strong solution to the following SDE (recall that $F$ is Lipschitz-continuous with respect to the $y$-variable)
\begin{align}\label{eq:Y_eps}
	Y_t^\varepsilon = Y_0 - \int_0^t F_s\big(Y^\varepsilon_s, Z_s,\widehat \sigma^2_s \big) \drm s + \int_0^t Z_s \cdot \drm X_s - \int_0^t \dot K^\varepsilon_s \drm s, \; t \in [0,T],
\end{align}
and define $\xi^\varepsilon \coloneqq U^{(\smalltext{-}1)}(Y_T^\varepsilon )$.
Following the arguments in \cite[Section 4.6, Step 1]{cvitanic2018dynamic}, one can verify that $(Y^\varepsilon,Z,K^\varepsilon)$ solves the 2BSDE \eqref{eq:agent_2BSDE} with the same generator $F$ but terminal condition $U(\xi^\varepsilon)$. Indeed, since $K^\varepsilon_T \leq K_T$, $K^\varepsilon$ satisfies the minimality condition and\footnote{This comes from the fact that $K^\eps$ inherits the integrability of $K$, $Z$ and $K$ are in the appropriate integrability class as solutions to the 2BSDE, $F$ is uniformly Lipschitz-continuous in $y$, and standard estimates for SDEs.}
\begin{align*}
	\sup_{\P \in \Pc} \E^\P \Big[ \big| U(\xi^\varepsilon) \big|^p \Big] < + \infty.
\end{align*}

One can then derive the appropriate integrability conditions on $(Y^\varepsilon, Z^\varepsilon)$ using \emph{a priori} estimates for 2BSDEs, see for example \citeayn[Theorem 4.4]{possamai2018stochastic}, to obtain
\begin{align}\label{eq:2BSDE_estimate}
	\| Z \|^{\bar p}_{\H^\smalltext{\bar p}} +
	\| Y^\varepsilon \|^{\bar p}_{\D^\smalltext{{\bar p}}} < + \infty, \; \text{for some} \; \bar p \in (1,p).
\end{align}

\textit{Step 3.} We now want to prove that $\xi^\varepsilon$ introduced above belong to the relevant class of contracts in the sense of \Cref{def:relevant_contracts}, namely $\xi^\varepsilon \in \Cc$. Recall first that $\psi \in \Psi$ is such that \ref{thm:cond_a} is satisfied. In particular, for $t \in [0,T]$ and $(x,y,z,S) \in \Omega \times \R \times \R^d \times  {\bf \Sigma}_t(x)$, the map $E\ni\gamma  \longmapsto \psi_t(x,y,z,\gamma,S)-F_t(x,y,z,S) \in \R_{\smallertext{+}}$ is surjective onto $[0,+\infty)$. Since $\dot K^\varepsilon_s \geq 0$ for all $s \in [0,T]$, the surjectivity of the above function and classical measurable selection arguments allow to deduce the existence of an $\F$-predictable $\Gamma^{\psi, \varepsilon}$ such that 
\begin{align}\label{eq:K_as_psi}
	\dot K^\varepsilon_s = \psi_s\big(Y^\varepsilon_s, Z_s,\Gamma^{\psi,\varepsilon}_s,\widehat \sigma^2_s \big) - F_s\big(Y^\varepsilon_s, Z_s,\widehat \sigma^2_s \big),\; \mathrm{d}s\otimes\P\text{\rm--a.e.}
\end{align}
Substituting in \eqref{eq:Y_eps}, we obtain
\begin{align*}
	Y_t^\varepsilon = Y_0 - \int_0^t  \psi_s\big(Y^\varepsilon_s, Z_s,\Gamma^{\psi,\varepsilon}_s,\widehat \sigma^2_s \big) \drm s + \int_0^t Z_s \cdot \drm X_s, \; t \in [0,T].
\end{align*}
In other words, $\xi^\varepsilon \coloneqq U^{(\smalltext{-}1)}(Y_T^\varepsilon )$ admits the representation \eqref{eq:xi_2BSDE_psi} for an initial value $Y_0 \in \R$ and a pair $(Z,\Gamma^{\psi,\varepsilon})$. 

\medskip

To be perfectly rigorous, the value $Y_0$ from the 2BSDE is $\Fc_0^\smallertext{+}$-measurable and thus in general not deterministic. However, it is standard to show that the principal can achieve the same utility by using either $Y_0$ or $\E^{\P^\smallertext{\star}}[Y_0]$ (see for instance similar arguments in \citeayn[Lemma 4.8]{hubert2022incentives} or \cite[Step 2 in the proof of Theorem 3.10]{chiusolo2026new}). Hence, the restriction to deterministic $Y_0$ is without loss of generality.

\medskip

To conclude that $\xi^\varepsilon \in \Cc$, it remains to verify that $(Z,\Gamma^{\psi,\varepsilon}) \in \Vc^\psi$. First, the appropriate integrability conditions on $(Y^\varepsilon, Z^\varepsilon)$ required in \Cref{def:contrat_new} $(i)$ are already established above in \eqref{eq:2BSDE_estimate}. Moreover, using the arguments in \textit{Step 1}, $K_\smallertext{T} = K^\varepsilon_\smallertext{T} = 0$, $\P^\star$--a.s., implying that $Y^\varepsilon = Y$, $\P^\star$--a.s. In particular, $\P^\star$ is also a best response to $\xi^\varepsilon$, namely $\P^\star \in \Pc^\star(\xi^\varepsilon)$, and thus \Cref{def:contrat_new} $(ii)$ is verified. The two above arguments ensure that $(Z,\Gamma^{\psi,\varepsilon}) \in \Vc^\psi$, hence $\xi^\varepsilon \in \Cc$. In other words, any contract $\xi \in \Xi$ can be approximated by a contract $\xi^\varepsilon \in \Cc$.

\medskip

\textit{Step 4.} To conclude the proof, we show that the utility of the principal remains unchanged whether the contract $\xi \in \Xi$ or its approximation $\xi^\varepsilon \in \Cc$ is implemented. By the above arguments, $\P^\star$ is a best response to both contracts, and since $Y^\varepsilon = Y$ $\P^\star$--a.s., we also have $\xi = \xi^\varepsilon$ $\P^\star$-a.s.  Thus
\begin{align*}
	\E^{\P^{\smalltext{\star}}} \bigg[ \Kc^{\smallertext{\rm P},\P^{\smalltext{\star}}}_\smallertext{T}  U_{\smallertext{\rm P}} ( \xi ) - \int_0^\smallertext{T} \Kc^{\smallertext{\rm P},\P^{\smalltext{\star}}}_s c_s^{\smallertext{\rm P}} \big(\beta_s^{\P^{\smalltext{\star}}} \big) \mathrm{d}s\bigg]
	= \E^{\P^{\smalltext{\star}}} \bigg[ \Kc^{\smallertext{\rm P},\P^{\smalltext{\star}}}_\smallertext{T}  U_{\smallertext{\rm P}} (\xi^\varepsilon) - \int_0^\smallertext{T} \Kc^{\smallertext{\rm P},\P^{\smalltext{\star}}}_s c_s^{\smallertext{\rm P}} \big(\beta_s^{\P^{\smalltext{\star}}} \big) \mathrm{d}s\bigg],
\end{align*}
in other words $J_{\smallertext{\rm P}} (\xi, \P^\star ) = J_{\smallertext{\rm P}} (\xi^\varepsilon, \P^\star )$. By arbitrariness of $\xi \in \Xi$, we conclude that $\overline V^\psi_{\smallertext{\rm P}} \geq V_{\smallertext{\rm P}}$.

\begin{remark}\label{rk:2BSDE}
		In the above proof, if instead of $\psi \in \Psi$, one wants to consider a different parametrisation function $\phi \in \Psi$ also satisfying \ref{thm:cond_a}, the only change is the representation of $\dot K^\varepsilon$ through \eqref{eq:K_as_psi} using $\phi$ for some $\F$-predictable process $\Gamma^{\phi,\varepsilon}$, potentially different from $\Gamma^{\psi,\varepsilon}$. In other words, the choice of the parametrisation function $\psi$ only impacts the way $\dot K^\varepsilon$ is represented, but all representations are clearly equivalent. This highlights in particular that the value function $\overline V^\psi_{\smallertext{\rm P}}$ is necessarily independent of the choice of $\psi \in \Psi$.
		
		\medskip
		
		Beyond this, one could actually relax {\rm \Cref{cond:wlog2}} by changing the approximation \eqref{eq:approx_K} of the {\rm 2BSDE} term $K$. For example, if considering $\psi \equiv F$, one can actually take $K^\varepsilon \equiv 0$. With this choice, $K^\varepsilon$ trivially satisfies the minimality condition, and the rest of the proof proceeds $($and can even be simplified$)$.
\end{remark}

\begin{remark}\label{rk:strong_duality}
As mentioned in {\rm \cite[Remark 3.15]{chiusolo2026new}}, the proof of {\rm \cite[Theorem 3.6]{cvitanic2018dynamic}} may fail if {\rm \Cref{ass:duality}} is not satisfied. More precisely, to derive the contract form \eqref{eq:contract_CPT} starting from \eqref{eq:xi_2BSDE}, one needs to replace \eqref{eq:K_as_psi} by existence of an appropriate process $\Gamma$ satisfying
	\begin{align}\label{eq:repres_Kdot}
		\dot K_s + F_s \big(Y_s, Z_s, \widehat \sigma^2_s \big) = H_s \big(Y_s, Z_s, \Gamma_s \big) - \dfrac12 {\rm Tr} \big[ \Gamma_s \widehat \sigma^2_s \big],\; \mathrm{d}s\otimes\P\textnormal{--a.e.},\; \P\in\Pc.
	\end{align}

	A sufficient condition for this is to verify that for all $(t,x,y,z) \in [0,T] \times \Omega \times \R \times \R^d$ and $S \in {\bf \Sigma}_t(x)$, the map
\begin{align*}
 \Sc_d(\R) \ni\gamma\overset{\varphi}\longmapsto H_t(x,y,z,\gamma) - F_t\big(x,y,z,S \big) - \dfrac12 {\rm Tr} \big[S \gamma \big] \in \R_\smallertext{+},
\end{align*}
is surjective. In {\rm \cite[Equation (4.17)]{cvitanic2018dynamic}}, this function is claimed to be non-negative, convex, continuous and coercive, and thus surjective onto $(0,\infty)$. Non-negativity and continuity are clear,  however, the map $\varphi$ is in general surjective \emph{only} onto $(\underline \varphi, \overline \varphi)$, where
\begin{align*}
	\underline \varphi \coloneqq  \inf_{\gamma \in \smallertext{\Sc}_\smalltext{d}(\R)} \varphi(\gamma), \; \text{and} \; \overline \varphi \coloneqq  \sup_{\gamma \in \smallertext{\Sc}_\smalltext{d}(\R)} \varphi(\gamma).
\end{align*}
While it can be checked that under mild assumptions $\overline \varphi = + \infty$, one needs to verify that $\underline \varphi = 0$, namely
\begin{align}\label{eq:dual_pb}
	F_t(x,y,z,S) = \inf_{\gamma \in \smallertext{\Sc}_\smalltext{d}(\R)} \bigg\{ H_t(x,y,z,\gamma) - \dfrac12 {\rm Tr} \big[S \gamma \big] \bigg\}, \; \text{\rm for} \; (t,x,y,z) \in [0,T] \times \Omega \times \R \times \R^d \; \text{\rm and} \; S \in {\bf \Sigma}_t(x).
\end{align}
Note that the above \Cref{eq:dual_pb} is equivalent to $-2F$ being the convex conjugate of $2H$. Since $2H$ is itself the convex conjugate of $-2F$, the previous condition is equivalent to $-2F = (2H)^\star = (-2F)^{\star \star}$, so $-2F$ coincides with its Fenchel bi-conjugate. Recall that, in general, for any appropriate function $g$ we have $g^{\star \star} \leq g$, implying here
\begin{align}\label{eq:duality2}
	F_t\big(x,y,z,S \big) \leq \inf_{\gamma \in \smallertext{\Sc}_\smalltext{d}(\R)} \bigg\{ H_t(x,y,z,\gamma) - \dfrac12 {\rm Tr} \big[S \gamma \big] \bigg\}.
\end{align}
However, the equality $g^{\star \star} = g$ holds if and only if $g$ is a proper, lower semi-continuous and convex function. In particular, {\rm \cite[Section 4.3]{chiusolo2026new}} exhibits a counterexample with a duality gap, meaning that the inequality \eqref{eq:duality2} can be strict for some specific choice of model coefficients. 
	
\medskip

Finally, note that {\rm \Cref{ass:duality}} is stronger than \eqref{eq:dual_pb}, requiring in addition that the infimum is achieved for some Borel-measurable function $\gamma$. This comes from the fact that the density $\dot K$ can vanish. Actually, under the probability measure $\P^\star$ corresponding to the agent's optimal response, one should have $\dot K \equiv 0$, $\P^\star$--{\rm a.s.} Therefore, one also needs to ensure existence of an appropriate process $\Gamma$ such that \eqref{eq:repres_Kdot} holds even when $\dot{K} = 0$. A sufficient condition for this is to assume in addition that the infimum in \eqref{eq:dual_pb} is achieved, which is exactly {\rm \Cref{ass:duality}}.\footnote{In \citeayn[Theorem 4.3 $(ii)$]{ren2023entropic}, the authors only require that the infimum is attained. However, this condition is not sufficient, since, as explained above, it is also required that this infimum coincides with $F$.}
\end{remark}

\subsubsection{The BSDE approach}\label{sss:BSDE_proof}

In this section, we use the approach recently developed in \cite{chiusolo2026new} to show optimality of the contract form \eqref{eq:xi_2BSDE_psi} when $\psi \in \Psi$ satisfies \ref{thm:cond_b}. Again, for this approach, it is required to assume that any admissible control for the agent induces a unique weak solution for the controlled SDE \eqref{eq:dynamic_SDE_P}, see \cite[Assumption 2.2]{chiusolo2026new}. 

\medskip
Since $V_{\smallertext{\rm P}}^\circ$ defined in \eqref{eq:pb_principal_FB} is the best possible value for the principal, showing that contracts of the form \eqref{eq:xi_2BSDE_psi} allow to achieve such value would be sufficient to ensure their optimality. For this, it is enough to show that for any pair $(\Sigma,Z)$ parametrising the contract form \eqref{eq:contract_FB} and a corresponding best response $\mathfrak{b}^\circ$ in the sense of \eqref{eq:u*_FB} above, one can find a triple $(Z,\Gamma,\psi)$ parametrising the contract form \eqref{eq:xi_2BSDE_psi} such that the optimal response coincides, as well as the dynamics of both contracts under the optimal response. We actually provide below a stronger result, namely that for a fixed parametrisation function $\psi \in \Psi$ satisfying \ref{thm:cond_b}, one can always find a pair $(Z,\Gamma) \in \Vc^\psi$ admissible for the contract form \eqref{eq:xi_2BSDE_psi}.

\begin{proposition}\label{prop:one-to-one}
	Fix $y \geq R_{\smallertext{\rm A}}$ and $\psi \in \Psi$ satisfying \ref{thm:cond_b}. Let $(Z,\Sigma) \in \Vc^\circ$ with a corresponding contract $\xi^\circ$ and $\P^\circ \in \Pc^{\circ,\star} (Z,\Sigma)$. Using the map $\gamma^\star$ defined in \eqref{eq:gamma_star}, there exists a unique strong solution $Y^\star$ to the following {\rm SDE}
		\begin{align}\label{eq:Y_star}
			Y_t^\star = y - \int_0^t \psi_s\big(Y^\star_s, Z_s, \gamma^\star_s(Y_s^\star,Z_s,\Sigma_s),\widehat \sigma^2_s \big) \drm s + \int_0^t Z_s \cdot \drm X_s, \; t \in [0,T].
		\end{align}
		Letting $\Gamma_\cdot\coloneqq \gamma^\star_\cdot(Y_\cdot^\star,Z_\cdot,\Sigma_\cdot)$, $Y^\star$ satisfies the representation \eqref{eq:xi_2BSDE_psi}, and $(Z,\Gamma) \in \Vc^\psi$. Moreover, $\P^\circ \in \Pc_\psi^{\star} (Z,\Gamma)$ and 
		\begin{align*}
			J_{\smallertext{\rm P}} ( \xi^\circ,\P^\circ) = J_{\smallertext{\rm P}} \big( U^{(\smalltext{-}1)} (Y_\smallertext{T}^{\star}),\P^\circ \big).
		\end{align*}
\end{proposition}

The above result states that, starting from an arbitrary pair $(Z,\Sigma) \in \Vc^\circ$ and an associated optimal response $\P^\circ \in \Pc^{\circ,\star} (Z,\Sigma)$ in the \textit{contractible-volatility} problem, one can actually find an appropriate contract $\xi \in \Cc$ in the sense of {\rm \Cref{def:relevant_contracts}} allowing to incentivise the exact same optimal response $\P^\circ$ but in the original problem, and such that the principal's expected utility remains the same. This ensures the required inequality for the principal's value, namely $\overline V^\psi_{\smallertext{\rm P}} \geq V^\circ_{\smallertext{\rm P}}$, and since $V^\circ_{\smallertext{\rm P}} \geq V_{\smallertext{\rm P}} \geq \overline V^\psi_{\smallertext{\rm P}}$, we deduce the equality between the three values.
	
\begin{remark}
	One may want to study the converse correspondence, namely defining the pair $(Z,\Sigma) \in \Vc^\circ$ corresponding to a given couple $(Z,\Gamma) \in \Vc^\psi$ and forcing a specific probability $\P^\star \in \Pc_\psi^{\star} (Z,\Gamma)$. For this, we need to ensure existence and uniqueness of a strong solution to the following {\rm SDE}
	\begin{align*}
		Y_t^\circ = y - \int_0^t F_s\big(Y^\circ_s, Z_s,  [\sigma \sigma^\top]_s \big(b^{\star,\psi}_s (Y^\circ_s,Z_s,\Gamma_s) \big) \big) \drm s + \int_0^t Z_s \cdot \drm X_s, \; t \in [0,T].
	\end{align*}
	Letting $\Sigma_\cdot \coloneqq  [\sigma \sigma^\top]_\cdot \big(b^{\star,\psi}_\cdot(Y^\circ_\cdot,Z_\cdot,\Gamma_\cdot) \big)$, or equivalently
	\begin{align*}
		\Sigma_t \in \argmin_{\smallertext{S} \in \smallertext{{\bf \Sigma}}_\smalltext{t}} \Big\{ \psi_t \big(Y^\circ_t,Z_t,\Gamma_t,S\big) - F_t\big(Y^\circ_t,Z_t,S\big) \Big\}, \; \drm t \otimes \P\text{\rm--a.e. on } [0,T] \times \Omega,
	\end{align*}
	$Y^\circ$ satisfies the representation \eqref{eq:contract_FB} with $(Z,\Sigma) \in \Vc^\circ$, and one can show equivalence between the agent's optimal response and equality between the principal's expected utility, following a similar reasoning as in the proof below.
\end{remark}

\begin{proof}[Proof of Proposition \ref{prop:one-to-one}]
	We fix throughout the proof $y \geq R_{\smallertext{\rm A}}$ and $\psi \in \Psi$ satisfying \ref{thm:cond_b}. Let $\Sigma \in \Sc$, $Z \in \Vc^\circ(\Sigma)$, and an associated terminal payment $\xi^\circ \coloneqq U^{(\smalltext{-}1)}(Y_\smallertext{T}^{y,Z,\circ})$ with $Y^{y,Z,\circ}$ defined through \eqref{eq:contract_FB}. As mentioned in \Cref{sss:BSDE}, such contract leads to a best-response $\mathfrak{b}^\circ_\cdot$, characterised through the function $b^\circ$ in \eqref{eq:u*_FB}. Our goal is to determine an appropriate parametrisation of a contract of the form \eqref{eq:xi_2BSDE_psi}, in the sense that such contract will lead to the same best response for the agent and the same dynamics for the principal's state variable. 
	
	\medskip
	As in the statement of \Cref{prop:one-to-one}, we consider the strong solution $Y^\star$ to the {\rm SDE} \eqref{eq:Y_star}, noting that \Cref{ass:psi_star} is sufficient to ensure that such solution exists and is unique. Letting $\Gamma_\cdot \coloneqq \gamma^\star_\cdot(Y_\cdot^\star,Z_\cdot,\Sigma_\cdot)$, one can rewrite $Y^\star$ as
	\begin{align*}
		Y_t^\star = y - \int_0^t \psi_s\big(Y^\star_s, Z_s, \Gamma_s, \widehat \sigma^2_s \big) \drm s + \int_0^t Z_s \cdot \drm X_s,\; t\in[0,T],
	\end{align*}
	which coincides with the representation \eqref{eq:xi_2BSDE_psi}. Furthermore, the growth condition in \Cref{ass:psi_star} ensures that $(Z,\Gamma) \in \Vc^\psi$ in the sense of \Cref{def:contrat_new}. By \Cref{prop:agent}, we deduce that the agent's optimal response takes the form $\mathfrak{b}^{\star,\psi}_\cdot \coloneqq b^{\star,\psi}_\cdot(Y^{\star}_\cdot,Z_\cdot,\Gamma_\cdot)$, where $b^{\star,\psi}$ satisfies \eqref{eq:optimum}. In other words
	\begin{align*}
		\mathfrak{b}^{\star,\psi}_t \in \argmax_{b \in \smallertext{B}_\smalltext{t}(\smallertext{\Sigma}_\smalltext{t}^\smalltext{\star, \psi})} \big\{ f_t \big(Y_t^\star,Z_t,b \big) \big\}, \; \text{with} \;
		\Sigma_t^{\star, \psi} \in \argmin_{\smallertext{S} \in \smallertext{\bf \Sigma}_\smalltext{t}} \big\{ \psi_t \big(Y_t^\star,Z_t,\Gamma_t,S \big) - F_t \big(Y_t^\star,Z_t,S \big) \big\}.
	\end{align*}
	Recalling that $\Gamma_\cdot = \gamma^\star_\cdot(Y_\cdot^\star,Z_\cdot,\Sigma_\cdot)$, with $\gamma^\star$ defined by \eqref{eq:gamma_star}, we have
	\begin{align*}
		 \psi_\cdot \big(Y_\cdot^\star,Z_\cdot,\Gamma_\cdot,\Sigma_\cdot \big) = \min_{\gamma \in \smallertext{E}} \psi_\cdot \big(Y_\cdot^\star,Z_\cdot,\gamma,\Sigma_\cdot \big) = F_\cdot \big(Y_\cdot^\star,Z_\cdot,\Sigma_\cdot \big).
	\end{align*}
	In particular, for such a choice of $\Gamma$, the infimum is achieved in \Cref{cond:revealing} for $\Sigma_\cdot \in {\bf \Sigma}_\cdot$. One can thus select $\Sigma_\cdot^{\star, \psi} = \Sigma_\cdot$, ensuring then that $\mathfrak{b}^{\star,\psi}_t\cdot= b^{\circ}_\cdot (Y_\cdot^\star,Z_\cdot,\Sigma_\cdot)$ since by \eqref{eq:u*_FB}
	\begin{align*}
		b^{\circ}_\cdot(Y_\cdot^\star,Z_\cdot,\Sigma_\cdot) \in \argmax_{b \in \smallertext{B}_\smalltext{\cdot}(\Sigma_\smalltext{\cdot})} \big\{ f_\cdot (Y_\cdot^\star,Z_\cdot,b) \big\}.
	\end{align*}
	Finally, since $\psi \in \Psi$ satisfies in particular \Cref{cond:revealing}, the process $Y^\star$ satisfies, under the probability measure $\P^\star$ corresponding to the above optimal effort, the same SDE as in \eqref{eq:contract_FB}.
\end{proof}

\begin{remark}\label{rk2:psi=0}
		In the above {\rm BSDE}-based proof of {\rm \Cref{thm:principal}}, we impose {\rm \Cref{ass:psi_star}} on $\psi$, so that the generator in \eqref{eq:xi_2BSDE_psi} remains Lipschitz-continuous in the $y$-variable and with appropriate growth in $y$, thus ensuring that strong existence and uniqueness hold, and classical estimates yield $Y^{\star}\in \D^p$ for the same exponent $p>1$ as in {\rm \Cref{def:contract_FB}}. First, this condition is only a \emph{sufficient} condition ensuring that $(Z,\Gamma)\in \Vc^\psi$ whenever $(Z,\Sigma)\in \Vc^\circ$, the {\rm BSDE} argument may still apply under more general frameworks. Moreover, to prove the equality between the values of the \emph{contractible-volatility} and the original problem, namely $V^\circ_{\smallertext{\rm P}} = V_{\smallertext{\rm P}}$, it suffices to notice that $\psi \equiv F$ satisfies {\rm \Cref{ass:psi_star}}. 
\end{remark}

\section{Specific examples of contracts}\label{sec:contract_examples}

The contract form introduced in \Cref{def:contrat_new} remains very general, since the parametrisation can be any function $\psi \in \Psi$. While some of these conditions are usually easy to verify, it is useful to derive sufficient criteria ensuring they hold, and to provide explicit examples of functions $\psi$. First, by \Cref{cond:revealing}, we can write for some other map $\overline \psi$ 
\begin{align}\label{eq:overline_psi}
	\psi_t (x,y,z,\gamma,S)- F_t (x,y,z,S) \eqqcolon \overline \psi_t (x,y,z,\gamma,S) - \inf_{S^\smalltext{\prime} \in {\bf \Sigma}_\smalltext{t}(x)} \overline \psi_t (x,y,z,\gamma,S^\prime),
\end{align}
for all $(t,x,y,z,\gamma,S) \in [0,T] \times \Omega \times \R \times \R^d \times E \times \Sc^{\smallertext{+}}_d(\R)$. Through this construction, $\psi-F$ is non-negative for $S \in {\bf \Sigma}_t(x)$ and its infimum on $S \in {\bf \Sigma}_t(x)$ is naturally $0$, so that \Cref{cond:revealing} is automatically satisfied. 

\medskip

Next, we give a sufficient condition for the `infimum part' of \Cref{cond:wlog1}, namely
\begin{align}\label{cond:wlog_inf}
	\inf_{\gamma \in \smallertext{E}} \psi_t(x,y,z,\gamma,S) = F_t(x,y,z,S), \; \text{for all} \; (t,x,y,z,S) \in [0,T] \times \Omega \times \R \times \R^d \times \Sc^{\smallertext{+}}_d(\R).
\end{align}
The proof is immediate, we therefore omit it.
\begin{lemma}\label{lemma:lemma2}
	Let $E$ be arbitrary, and assume that we can write
	\[
	\overline\psi_t(x,y,z,\gamma,S)=\ell_t(x,y,z,S)-\eta_t(x,y,z,\gamma,S), \; \text{for all} \; (t,x,y,z,\gamma,S) \in [0,T] \times \Omega \times \R \times \R^d \times E \times \Sc^{\smallertext{+}}_d(\R),
	\]
	where $\eta$ takes values in $\R\cup\{+\infty\}$, and for any $(t,x,y,z)$, the map $\mathbf{\Sigma}_t(x)\ni S\longmapsto \ell_t(x,y,z,S)\in\R\cup\{+\infty\}$ is $\eta$-convex.\footnote{Recall that $\eta$-convexity stems from optimal transport theory, see \cite[Chapter 5, pages 54--57]{villani2009optimal}. Here, it means that for any $(t,x,y,z)$, the map $\mathbf{\Sigma}_t(x)\ni S\longmapsto \ell_t(x,y,z,S)\in\R\cup\{+\infty\}$ is not identically $+\infty$ and there exists $E\ni\gamma\longmapsto\zeta_t(x,y,z,\gamma)\in\R\cup\{+\infty,-\infty\}$ such that
	\[
	\ell_t(x,y,z,S)=\sup_{\gamma\in \smallertext{E}}\{\zeta_t(x,y,z,\gamma)-\eta_t(x,y,z,\gamma,S)\}, \; S\in\mathbf{\Sigma}_t(x).
	\]
	A function is {$\eta$}-convex if and only if it is equal to its double {$\eta$}-transform, that is
	\[
	\ell_t(x,y,z,S)=\sup_{\gamma\in \smallertext{E}}\inf_{\smallertext{S}^\smalltext{\prime}\in\smallertext{\mathbf{\Sigma}}_\smalltext{t}(x)}\{\ell_t(x,y,z,S^\prime)+\eta_t(x,y,z,\gamma,S^\prime)-\eta_t(x,y,z,\gamma,S)\}\eqqcolon \ell^{\eta\eta}_t(x,y,z,S).
	\]} Then, {\rm\Cref{cond:wlog_inf}} holds.
\end{lemma}


It remains to check that the infimum is attained in \Cref{cond:wlog_inf}, which is the case for instance if $E=\Sc_d(\R)$, $\eta$ is taken as proportional to the standard scalar product there---in which case $\eta$-convexity is simply standard convexity plus being proper and lower--semi-continuous---, and $\ell$ is in addition strictly convex in $\gamma$.

\medskip
While the above highlights sufficient conditions on the function $\overline{\psi}$ from which one can construct an appropriate function $\psi$ through \eqref{eq:overline_psi}, the construction remains rather implicit. We therefore present in the following sections explicit functions $\psi \in \Psi$.

\subsection{Standard cases}\label{ss:case_duality}

{In this section, we first consider two \textit{standard cases}: when only the drift is controlled, and when the volatility is controlled but \Cref{ass:duality} holds. Both cases have been fully solved in the existing literature, and our objective here is merely to exhibit a suitable function $\psi$ so as to recover the optimal contract form already established therein.
	
\medskip
First, in a framework where only the drift is controlled, it is well-known that the optimal contract takes the form $\xi = U^{(\smalltext{-}1)}\big(Y^{y,Z}_\smallertext{T}\big)$ for a constant $y \in \R$ and an appropriate process $Z$, namely satisfying the usual integrability conditions and such that the process $Y$ is the well-defined solution of the following SDE
\begin{align*}
	Y^{y,Z}_t &= y - \int_0^t F_s \big(Y^{y,Z}_s, Z_s,\widehat \sigma^2_s \big) \drm s 
	+ \int_0^t Z_s \cdot \drm X_s, \; t \in [0,T], \; \P\textnormal{--a.s.}
\end{align*}
Note that in such framework, the density of the quadratic variation is uncontrolled, meaning that for all $(t,x) \in [0,T] \times \Omega$, the set $\Sigma_t(x)$ is reduced to the singleton $\{ [\sigma \sigma^\top]_t (x)\}$. Consequently, the `constrained' Hamiltonian $F$ defined by \eqref{eq:hamiltonian_constrained} boils down to the usual Hamiltonian for drift control
\begin{align*}
	F_t (x,y,z,S) = \sup_{b \in \smallertext{B}} \big\{ \sigma_t (x)  \lambda_t(x,b) \cdot z - c_t(x,b) - k_t(x,b) y \big\}, \; (t,x,y,z) \in [0,T] \times \Omega \times \R \times \R^d, \; S \in \Sigma_t(x),
\end{align*}
and is thus independent of the $S$-variable. It is straightforward to see that taking $\psi \equiv F$ in the general contract form introduced in \Cref{def:relevant_contracts} allows to recover the above contract. This choice naturally satisfies \Cref{cond:revealing}, as well as \Cref{cond:wlog1} with minimum achieved for any $\gamma \in E$. However, this natural choice does not satisfy \Cref{cond:wlog2}, as the supremum of $\psi$ is clearly $F$ instead of $\infty$. Nevertheless, the existing literature on principal--agent problems with drift control actually always assumes weak uniqueness for the controlled SDE \eqref{eq:dynamic_SDE_P}. Therefore, the optimality of the contract form can be proved using the BSDE-approach as in \Cref{sss:BSDE_proof} and, as highlighted before, the first part of \Cref{cond:wlog2} is actually not needed for this approach. We emphasise that, without the assumption of weak uniqueness, the representation of the optimal contract would involve an additional term, namely an orthogonal martingale, and it remains unclear in general settings whether such a term can be taken to be zero without loss of generality. For further details on this technical point, we refer to \cite[Remarks 2.3 and 3.3]{chiusolo2026new}, as well as the general approach in \citeayn{krsek2026randomisation}.} 

\medskip
In a framework with volatility control, and where \Cref{ass:duality} is satisfied, the approach introduced in \cite{cvitanic2018dynamic} is valid, implying that the original contract form \eqref{eq:contract_CPT} is optimal. This form can be recovered from our more general form of contracts introduced in \Cref{def:relevant_contracts} by taking 
\begin{align}\label{eq:psi_contract_CPT}
	\psi_t(x,y,z,\gamma,S) = H_t (x,y,z,\gamma) - \dfrac12 {\rm Tr} \big[ \gamma S \big], \; (t,x,y,z,\gamma,S) \in [0,T] \times \Omega \times \R \times \R^d \times \Sc_d(\R) \times\Sc^{\smallertext{+}}_d(\R).
\end{align}
Indeed, one can first compute that, by definition of $H$ or directly using \eqref{eq:link_hamiltonian}
\begin{align*}
	\inf_{S \in {\bf \Sigma}_\smalltext{t}(x)} \big\{ \psi_t(x,y,z,\gamma,S)  - F_t (x,y,z,S) \big\}
	= H_t (x,y,z,\gamma) - \sup_{S \in {\bf \Sigma}_\smalltext{t}(x)} \bigg\{ F_t (x,y,z,S) + \dfrac12 {\rm Tr} \big[ \gamma S \big] \bigg\}
	= 0,
\end{align*}
thus \Cref{cond:revealing} is satisfied. Still using the definition of $H$, we also have that for all $S^\prime \in {\bf \Sigma}_t(x)$
\begin{align*}
	H_t(x,y,z,\gamma) \geq F_t (x,y,z,S^\prime) + \dfrac12 {\rm Tr} \big[ \gamma S^\prime \big],
\end{align*}
so that provided that for any $(t,x,S)$, we can always find $S^\circ_t(x)\in{\bf \Sigma}_t(x)$ such that $S^\circ_t(x)\prec S$ $($or $S\prec S^\circ_t(x))$, where $\prec$ is the standard---strict---partial order on the set of symmetric semi-definite, $d\times d$ matrices with real entries, we have
\begin{align*}
	\sup_{\gamma \in \smallertext{\Sc}_\smalltext{d}(\R)} \psi_t(x,y,z,\gamma,S) \geq F_t(x,y,z,S^\prime) + \dfrac12 \sup_{\gamma \in \smallertext{\Sc}_\smalltext{d}(\R)} \big\{  {\rm Tr} \big[ \gamma (S^\prime-S) \big] \big\} = + \infty.
\end{align*}
Finally, under \Cref{ass:duality}, there exists a Borel-measurable function $\gamma^\star$ such that 
\begin{align*}
	\psi_t(x,y,z,\gamma^\star_t(x,y,z,S),S) = H_t  \big(x, y, z, \gamma^\star_t(x,y,z,S) \big) - \dfrac12 {\rm Tr} \big[ \gamma^\star_t(x,y,z,S) S \big] = F_t(x,y,z,S).
\end{align*}
From the two equations above, we conclude that $\psi$ also satisfies \Cref{cond:wlog1,,cond:wlog2}.

\medskip
Our objective is now to construct an explicit function $\psi$ for the general case, so as to ensure the existence of an optimal contract when the volatility is controlled but \Cref{ass:duality} does not hold. In the following sections, we present two natural choices for $\psi$, derived respectively from the `contractible-volatility' and 2BSDE approaches.

\subsection{Duality contracts}\label{ss:case_hamiltonian}

When \Cref{ass:duality} does not hold, one can \textit{slightly} correct the original contract form \eqref{eq:contract_CPT} introduced in \cite{cvitanic2018dynamic}. We have seen in the previous section that a natural choice for $\psi$ is \eqref{eq:psi_contract_CPT},
or equivalently in the form \eqref{eq:overline_psi} with $\overline \psi$ defined as
\begin{align*}
	\overline \psi_t(x,y,z,\gamma,S) \coloneqq - F_t (x,y,z,S) - \dfrac12 {\rm Tr} \big[ \gamma S \big], \; (t,x,y,z,\gamma,S) \in [0,T] \times \Omega \times \R \times \R^d \times \Sc_d(\R) \times\Sc^{\smallertext{+}}_d(\R).
\end{align*}
This construction works under \Cref{ass:duality} since, in this case, $F$ coincides (abusing notations slightly) with its bi-conjugate $F^{\star \star}$, defined for all $(t,x,y,z) \in [0,T] \times \Omega \times \R \times \R^d$ by
\begin{align}\label{eq:bi-conjugate}
	F^{\star \star}_t ( x, y, z,S) \coloneqq &\ \inf_{\gamma \in \smallertext{\Sc}_\smalltext{d}(\R)} \bigg\{ H_t ( x, y, z, \gamma) - \dfrac12 {\rm Tr} \big[ \gamma S \big] \bigg\}.
\end{align} 
While the above link breaks when \Cref{ass:duality} does not hold, the intuition is to define $\overline \psi$ with $F^{\star \star}$ introduced above instead of $F$, which will then naturally coincide with its bi-conjugate, restoring duality.

\medskip
	
With this in mind, we now specify the parametrisation function $\psi$ via \eqref{eq:overline_psi}, with $\overline{\psi}$ defined using $F^{\star \star}$ in \eqref{eq:bi-conjugate} above
\[
\overline{\psi}_t(x,y,z,\gamma,S)\coloneqq -F^{\star \star}_t(x,y,z,S)-\frac12\mathrm{Tr}[\gamma S], \; (t,x,y,z,\gamma,S) \in [0,T] \times \Omega \times \R \times \R^d \times E \times \Sc^{\smallertext{+}}_d(\R).
\]
Since $H$ is convex---as the convex conjugate of $(-2F)$---it coincides with its bi-conjugate, and we deduce
\begin{align*}
	\inf_{\smallertext{S}^\smalltext{\prime} \in \smallertext{\bf \Sigma}_\smalltext{t}(x)} \overline \psi_t (x,y,z,\gamma,S^\prime) 
	= \inf_{\smallertext{S}^\smalltext{\prime} \in \smallertext{\bf \Sigma}_\smalltext{t}(x)} \bigg\{ - \inf_{\gamma \in \smallertext{\Sc}_\smalltext{d}(\R)} \bigg\{ H_t ( x, y, z, \gamma) - \dfrac12 {\rm Tr} \big[ \gamma S \big] \bigg\} - \frac12\mathrm{Tr}[\gamma S^\prime] \bigg\}
	= - H_t ( x, y, z, \gamma),
\end{align*}
for all $(t,x,y,z,\gamma,S) \in [0,T] \times \Omega \times \R \times \R^d \times E \times \Sc^{\smallertext{+}}_d(\R)$, so that
\begin{align*}
	\psi_t (x,y,z,\gamma,S) 
	&= F_t (x,y,z,S) -F^{\star \star}_t(x,y,z,S)-\frac12\mathrm{Tr}[\gamma S] + H_t ( x, y, z, \gamma).
\end{align*}

We are thus considering terminal payments of the form $\xi = U^{(\smalltext{-}1)}(Y_\smallertext{T})$ with
\begin{align}\label{eq:contract_Dylan}
	Y_\smallertext{T} &= y_0 - \int_0^\smallertext{T}  H_s( Y_s, Z_s, \Gamma_s)  \drm s 
	+ \int_0^\smallertext{T} Z_s \cdot \drm X_s 
	+ \dfrac12 \int_0^\smallertext{T} {\rm Tr} \big[ \Gamma_s \drm [ X ]_s \big]+\int_0^T\big( F^{\star\star} - F\big)_s \big( Y_s, Z_s, \widehat \sigma^2_s\big)\mathrm{d}s.
\end{align}
By construction of $\psi$ via \eqref{eq:overline_psi}, \Cref{cond:revealing} is naturally satisfied. Then, \Cref{lemma:lemma2} also applies here using standard convex duality. However, if $H$ fails to be strictly convex in $\gamma$, the infimum over $\gamma$ may not be attained. In this case, one may replace $H$ by any other strictly convex function, say $\overline H_t(x,y,z,\gamma)$,\footnote{A natural choice here would be a quadratic perturbation of the form $\overline H_t(x,y,z,\gamma)\coloneqq  H_t(x,y,z,\gamma)+\frac\eps2\mathrm{Tr}[\gamma\gamma^\top],\; \eps>0$.} and take instead
\[
\overline{\psi}_t(x,y,z,\gamma,S)=-\inf_{\gamma^\smalltext{\prime} \in\smallertext{\Sc}_\smalltext{d}(\R)}\bigg\{\overline H_t(x,y,z,\gamma^\prime)-\frac12\mathrm{Tr}[\gamma^\prime S]\bigg\}-\frac12\mathrm{Tr}[\gamma S]=-\overline H^\star_t(x,y,z,S)-\frac12\mathrm{Tr}[\gamma S].
\]
This leads to
\[
\psi_t(x,y,z,\gamma,S)\coloneqq  F_t (x,y,z,S) + \overline H^{\star\star}_t(x,y,z,S)-\overline H^\star_t(x,y,z,S)-\frac12\mathrm{Tr}[\gamma S],
\]
so that the infimum over $\gamma$ is still $0$ by Fenchel duality, but it is also attained by strict convexity. Finally, using the standard partial order on symmetric matrices, and assuming that for any $(t,x,S)$, we can always find $S_t^\circ(x)\in{\bf \Sigma}_t(x)$ such that $S^\circ_t(x)\prec S$ $($or $S\prec S_t^\circ(x))$, we obtain the desired behaviour at infinity for $\psi$. Under these assumptions, $\psi \in \Psi$ and the contract form \eqref{eq:contract_Dylan} above is optimal. 

\medskip

To interpret the new form of contract \eqref{eq:contract_Dylan}, recall first that, by \eqref{eq:duality2}, $F_t (x,y,z,S ) - F^{\star \star}_t (x,y,z,S ) \leq 0,$ with equality only if \Cref{ass:duality} is satisfied, so this term corresponds to an extra cost for the principal in the contract form \eqref{eq:contract_Dylan} compared to the original contract form \eqref{eq:contract_CPT}. The economic interpretation of this extra cost is as follows: in general, the only way the agent can always reach the value $F^{\star \star} $ when optimising is if he is allowed to randomise his controls, unless of course strong duality holds. Absent randomisation, the agent can only be guaranteed to reach the smaller value $F$. Hence, to properly align the agent's incentives, the principal must compensate him for this extra duality gap.
More precisely, write first
\[
F^{\star \star}_t ( x, y, z,S) =\inf_{\gamma \in \smallertext{\Sc}_\smalltext{d}(\R)}\sup_{\smallertext{S}^\smalltext{\prime}\in \smallertext{\mathbf{\Sigma}}_\smalltext{t}(x)}\bigg\{F_t ( x, y, z,S^\prime)+\frac12\mathrm{Tr}\big[\gamma(S^\prime-S)\big]\bigg\},
\]
and formally consider a randomised version of the above: assume for simplicity that $\mathbf{\Sigma}_t(x)$ is convex and compact (so that in particular the restriction of $\Sc_d(\R)$ to it is still a Polish space), let $\mathfrak S_t(x)$ and $\mathfrak S_d(\R)$ be the sets of probability measures on $\mathbf{\Sigma}_t(x)$ and $\Sc_d(\R)$ respectively, and define
\begin{align*}
\overline F^{\star\star}_t(x,y,z,S) &\coloneqq  \inf_{\gamma \in \smallertext{\mathfrak{S}}_\smalltext{d}(\R)}\sup_{\smallertext{S}^\smalltext{\prime}\in \smallertext{\mathfrak{S}}_\smalltext{t}(x)}\int_{\smallertext{\Sc}_\smalltext{d}(\R)}\int_{\smallertext{\mathbf{\Sigma}}_\smalltext{t}(x)}\bigg\{F_t ( x, y, z,\mathrm{s})+\frac12\mathrm{Tr}\big[g(S-\mathrm{s})\big]S^\prime(\mathrm{d}\mathrm{s})\gamma(\mathrm{d}g)\bigg\}.
\end{align*}
The function on the right-hand side is clearly bilinear with respect to $(\gamma,S^\prime)$, and since $\mathfrak S_t(x)$ is convex and compact (for the weak topology) and $\mathfrak S_d(\R)$ is convex, we can apply Sion's minimax theorem to deduce
\begin{align*}
\overline F^{\star\star}_t(x,y,z,S) &\coloneqq  \sup_{\smallertext{S}^\smalltext{\prime}\in \smallertext{\mathfrak{S}}_\smalltext{t}(x)}\inf_{\gamma \in \smallertext{\mathfrak{S}}_\smalltext{d}(\R)}\int_{\smallertext{\mathbf{\Sigma}}_\smalltext{t}(x)}\int_{\smallertext{\Sc}_\smalltext{d}(\R)}\bigg\{F_t ( x, y, z,\mathrm{s})+\frac12\mathrm{Tr}\big[g(S-\mathrm{s})\big]S^\prime(\mathrm{d}\mathrm{s})\gamma(\mathrm{d}g)\bigg\}\\
&= \inf_{\gamma \in \smallertext{\mathfrak{S}}_\smalltext{d}(\R)}\int_{\smallertext{\mathbf{\Sigma}}_\smalltext{t}(x)}F_t ( x, y, z,S)\gamma(\mathrm{d}g)= F_t(x,y,z,S).
\end{align*}
Strong duality always holds in this case (see \citeayn[Equation (1.11)]{buckdahn2014value} for similar arguments).

\subsection{Forcing contracts}\label{ss:case_forcing}

Alternatively, the `contractible‑volatility' approach developed in \cite{chiusolo2026new} makes natural the consideration of `forcing'---or `penalisation'---contracts, in the sense that such contracts should strongly incentivise the agent to choose efforts in order to achieve an optimal quadratic variation from the principal's point of view. With this in mind, we propose the following contracts $\xi = U^{(\smalltext{-}1)} (Y_\smallertext{T})$, where $Y$ is solution to 
\begin{align}\label{eq:contract_new_simple}
	Y_t = y_0 - \int_0^t F_s \big(Y_s,Z_s, \widehat \sigma^2_s \big) \drm s 
	- \int_0^t \| \widehat \sigma^2_s-\Gamma_s \|^2 \drm s
	+ \int_0^t Z_s \cdot \drm X_s 
	+ \int_0^t \inf_{\smallertext{S} \in \smallertext{\bf \Sigma}_\smalltext{s} } \big\| S - \Gamma_s \big\|^2 \drm s,
\end{align}
indexed by a pair of processes $(Z,\Gamma)$, satisfying appropriate integrability conditions derived from \Cref{def:contrat_new}, and where $\|A\|^2= {\rm Tr} [A^\top A]$ for any $A \in \Sc_d(\R)$. The contract form \eqref{eq:contract_new_simple} coincides with the general form \eqref{eq:xi_2BSDE_psi} under the specification that $E \coloneqq  \Sc_d(\R)$ and
\begin{align*}
	\psi_t(x,y,z,\gamma,S) \coloneqq F_t(x,y,z,S) + \| S-\gamma \|^2 - \inf_{\smallertext{S}^\smalltext{\prime} \in \smallertext{\bf \Sigma}_{\smalltext{t}}(x) } \big\| S^\prime - \gamma \big\|^2,
\end{align*}
for all $(t,x,y,z,\gamma,S) \in [0,T] \times \Omega \times \R \times \R^d \times\Sc_d(\R) \times \Sc^{\smallertext{+}}_d(\R) $, or equivalently using \eqref{eq:overline_psi}, $\overline{\psi}_t(x,y,z,\gamma,S)\coloneqq \| S-\gamma\|^2= \| S\|^2+ \| \gamma\|^2-2\mathrm{Tr}[S\gamma^\top].$ First, it is clear by definition that $\psi$ satisfies \Cref{cond:revealing}. Moreover, if $\gamma \in {\bf \Sigma}_t(x)$, then the infimum in \Cref{cond:revealing} is achieved for $S = \gamma$. To verify \Cref{cond:wlog1,cond:wlog2}, we notice that
	\begin{align*}
		\sup_{\gamma \in \Sc_\smalltext{d}(\R)} \psi_t(x,y,z,\gamma,S) 
		&\geq F_t(x,y,z,S) + \sup_{\gamma \in \smallertext{\Sc}_\smalltext{d}(\R)} \Big\{\|S-\gamma\|^2-\|S^{\prime\prime} - \gamma\|^2\Big\} \\
		&=F_t(x,y,z,S) + \|S\|^2-\|S^{\prime\prime}\|^2+\sup_{\gamma \in \smallertext{\Sc}_\smalltext{d}(\R)} \big\{2{\rm Tr}[(S^{\prime\prime}-S)^\top \gamma]\big\} \\
		&\geq F_t(x,y,z,S) +\|S\|^2-\|S^{\prime\prime}\|^2+2 \sup_{c \in \R} \big\{c \|S^{\prime\prime}-S\|^2\big\} = +\infty, 
	\end{align*}
	where $S^{\prime\prime}\in {\bf \Sigma}_{{t}}(x)$ and $S^{\prime\prime}\neq S$, as well as, for $S \in  {\bf \Sigma}_t(x)$
\begin{align*}
	\inf_{\gamma \in \Sc_\smalltext{d}(\R)} \psi_t(x,y,z,\gamma,S) 
	= F_t(x,y,z,S) + \inf_{\gamma \in \Sc_\smalltext{d}(\R)} \bigg\{ \| S-\gamma \|^2 - \inf_{S^\smalltext{\prime} \in {\bf \Sigma}_t(x) } \big\| S^\prime - \gamma \big\|^2 \bigg\}
	= F_t(x,y,z,S), 
\end{align*}
achieved again for $\gamma = S$. Alternatively, one can check that this corresponds to $\eta$-convexity in the sense of \Cref{lemma:lemma2} with $\eta_t(x,y,z,\gamma,S)\coloneqq 2\mathrm{Tr}[S\gamma^\top]-\| \gamma\|^2$ and the choice $\ell(S)=\| S\|^2$, which is $\eta$-convex since
\[
\| S\|^2=\sup_{\gamma\in \smallertext{\Sc}_\smalltext{d}(\R)}\big\{-2\|\gamma\|^2-2\mathrm{Tr}[S\gamma^\top]+\| \gamma\|^2\big\}.
\]
We highlight that this specification for $\psi$ also satisfies \Cref{ass:psi_star}, so the BSDE approach can be used to prove the optimality of this contract form.

\medskip
The contract form \eqref{eq:contract_new_simple} is quite natural when following the `contractible‑volatility' approach. Indeed, in the classical literature on contract theory, the usual technique to show that the \textit{first-best} value can be achieved when the principal observes the agent's control is to exhibit \textit{forcing} contracts, namely contracts that strongly penalise the agent if his control deviates from the effort recommended by the principal. Applying this idea here, with the `observable' control being the quadratic variation, naturally leads to the contract form \eqref{eq:contract_new_simple}, where the parameter $\Gamma$ chosen by the principal represents the density of the desired quadratic variation. More precisely, assuming to simplify that $\Gamma_t \in {\bf \Sigma}_t(X)$ for all $t \in [0,T]$, then the contract can be simplified as follows
\begin{align*}
	Y_t = y_0 - \int_0^t F_s \big(Y_s,Z_s, \widehat \sigma^2_s \big) \drm s 
	+ \int_0^t Z_s \cdot \drm X_s
	- \int_0^t \| \widehat \sigma^2_s-\Gamma_s \|^2 \drm s.
\end{align*}
Comparing with the usual form of contract for drift control, the main difference is the last integral term, which is non-negative and thus clearly penalising the agent whenever the density $\widehat \sigma^2$ of the observable quadratic variation deviates from the recommended one announced by the principal via the parameter $\Gamma$.

\section{Examples and counter-examples}\label{sec:examples}

In this section, we revisit some of the examples and counter-examples introduced in the recent literature on principal--agent problems with volatility control. More precisely, in \Cref{ss:example_ok} we highlight that most application-related examples in the literature do in fact satisfy \Cref{ass:duality}; see \citeayn{cvitanic2017moral}, \citeayn{aid2022optimal}, and \citeayn{baldacci2022governmental} described below. In contrast, the following section recalls the counter-example proposed in \cite[Section~4.3]{chiusolo2026new}, where \Cref{ass:duality} is not satisfied, and hence the contract form \eqref{eq:contract_CPT} does not allow one to achieve the `contractible-volatility' value, as demonstrated in \cite{chiusolo2026new}. We nevertheless show that the new contract form we propose attains this value. Finally, beyond the portfolio management example, \cite{cvitanic2017moral} introduces a relatively general model---under CARA utility functions and within a Markovian framework---which was subsequently extended in \cite{cvitanic2018dynamic}. However, despite additional convexity assumptions compared to \cite{cvitanic2018dynamic}, this more restrictive model may still fail to satisfy \Cref{ass:duality}, as illustrated by a counter-example in \Cref{ss:other}.
	
\subsection{Examples in the literature}\label{ss:example_ok}

We first recall several applied models from the literature and show that they satisfy \Cref{ass:duality}, thereby confirming that the use of the original contract form \eqref{eq:contract_CPT} in these works is fully justified and that their results remain valid.

\subsubsection{Delegated portfolio management}\label{sss:CPT2017}

%

In the example of \cite[Section~2]{cvitanic2017moral}, the output process admits the following linear dynamics
\begin{align*}
	\drm X_t = \big(v_1 \beta^1_t +v_{2} \beta^2_t \big)\drm t + \beta^1_t \drm W^1_t+ \beta^2_t \drm W^2_t, \; t \in [0,T],
\end{align*}
where $\beta \coloneqq (\beta^1,\beta^2)$ denotes the agent's control, taking values in $\R^2$, $W^1$ and $W^2$ are two independent Brownian motions, and $ (v_1,v_2)\in \R^2$ are fixed parameters. The agent's preferences are described by a CARA utility function and the following cost of effort, with constant parameters $(a_1,a_2)\in \R^2$ and $ (k_1,k_2)\in (0,+\infty)^2$,
\begin{align*}
	k(b) = \dfrac12 k_1 |b^1-a_1|^2 + \dfrac12 k_2 |b^2-a_2|^2, \; b \coloneqq (b^1,b^2)\in \R^{2}.
\end{align*}
With this framework in mind, we are led to define for all $z \in \R$ and $b \coloneqq (b^1,b^2) \in \R^2$
\begin{align*}
	f(z,b) \coloneqq z(v_1 b^1+v_2 b^2)-\dfrac12 k_1 |b^1-a_1|^{2}-\dfrac12^k_2 |b^2-a_2|^{2}.
\end{align*}
To introduce the agent's `constrained' Hamiltonian $F$, we first note that the quadratic variation of $X$ satisfies
\begin{align*}
	\drm [X]_t =  \big( |\beta_t^1|^2 + |\beta_t^2|^2 \big) \drm t, \; t \in [0,T].
\end{align*}
In other words, for $S \in \R_{\smallertext{+}}$ describing the density of the quadratic variation fixed by the principal, the agent is constrained to choose his effort $\beta \coloneqq (\beta^1,\beta^2)$ taking values on the circle of radius $\sqrt S$, characterised by
\[    
g(S,b)\coloneqq |b^1|^{2}+|b^2|^{2}-S=0, \; \text{for} \; b \coloneqq (b^1,b^2)\in \R^{2}, \; S \in \R_{\smallertext{+}}.
\] 
The agent's `constrained' Hamiltonian $F$ is thus defined for $(z,S) \in \R \times \R_{\smallertext{+}}$ by
\begin{align*}
	F(z,S)\coloneqq \sup_{\{b \in\R^\smalltext{2} : g(S,b)=0\}} \{ f(z,b)\}.
\end{align*}
To solve this constrained optimisation problem, we first introduce the corresponding Lagrangian
\begin{align*}
	\mathcal L(z,S,b,\lambda) \coloneqq f(z,b)-\lambda g(S,b), \; \text{for} \; (z, S, b, \lambda) \in \R \times \R_{\smallertext{+}} \times \R^{2} \times \R.
\end{align*}
Computing the partial derivatives of \(\mathcal L\) with respect to the variables \( (b^1,b^2,\lambda) \) gives the following first-order conditions
	\begin{align}\label{eq:FOC}
		z v_1 - k_1(b^1-a_1)-2\lambda b^1 =0, \; 
		z v_2 - k_2(b^2-a_2)-2\lambda b^2 =0, \;
		|b^1|^{2}+|b^2|^{2}-S = 0.
	\end{align}
Solving the first two equations for $(b^1,b^2) \in \R^2$, we get
\begin{equation}\label{eq:v_1v_2-lambda}
	b^{1,\star}(z,\lambda)=\dfrac{k_1 a_1+z v_1}{k_1+2\lambda},\;
	b^{2,\star}(z,\lambda)=\dfrac{k_2 a_2+z v_2}{k_2+2\lambda}.
\end{equation}
To ensure that this point is a maximiser, one can compute the second-order condition and derive the admissible range for \(\lambda\), namely $\Lambda\coloneqq \big(-1/2\min\{k_1,k_2\},+\infty\big).$ Now, plugging \eqref{eq:v_1v_2-lambda} in \eqref{eq:FOC}, we get 
\begin{equation}\label{eq:impl-eq}
	\phi (z,\lambda) = S, \; \text{with} \; \phi (z,\lambda) \coloneqq \bigg|\dfrac{k_1 a_1+z v_1}{k_1+2 \lambda} \bigg|^{2}+ \bigg|\dfrac{k_2 a_2+z v_2}{k_2+2\lambda} \bigg|^{2}, \; (z,\lambda,S) \in \R \times \R \times \R_{\smallertext{+}}.
\end{equation}
Differentiating $\phi (z,\lambda)$ with respect to $\lambda \in \R$, we obtain
\begin{align*}
	\partial_\lambda \phi(z, \lambda)=
	-\frac{4 |k_1 a_1+z v_1|^2}{(k_1+2\lambda)^{3}}
	-\frac{4 |k_2 a_2+z v_2|^2}{(k_2+2\lambda)^{3}}.
\end{align*}
Every denominator is strictly positive on \(\Lambda\); hence $\phi(z, \lambda)$ is decreasing with respect to $\lambda \in \Lambda$, with
\begin{align*}
	\lim_{\lambda\to\smalltext{-}1/2\mathrm{min}\{k_\smalltext{1},k_\smalltext{2}\}} \phi(z,\lambda)=+\infty,\;
	\lim_{\lambda\to+\infty} \phi(z,\lambda)=0.
\end{align*}
By continuity and strict monotonicity, we deduce that for fixed $(z,S) \in \R \times \R_{\smallertext{+}}$, there exists a unique solution $\lambda^\star(z,S)\in\Lambda$ to the implicit equation \eqref{eq:impl-eq}. The (unique) admissible optimal effort is thus given by
\begin{align*}
	b^\star(z,S) \coloneqq \big( b^{1,\star}(z,\lambda^\star(z,S)), \; b^{2,\star}(z,\lambda^\star(z,S)) \big), \; (z,S) \in \R \times \R_{\smallertext{+}},
\end{align*}
and the constrained Hamiltonian can be computed as $F(z,S) = f(z,b^\star(z,S))$.
To investigate the curvature of $F$, note that the parameter $S$ appears inside $\mathcal L$ only linearly as $\lambda S$, so the equality--constrained envelope theorem yields
\begin{align*}
	\partial_\smallertext{S} F(z,S)=\lambda^\star(z,S), \; \text{and} \; \partial_{\smallertext{S}\smallertext{S}} F(z,S)= \partial_S \lambda^\star(z,S) \; (z,S) \in \R \times \R_{\smallertext{+}}.
\end{align*}
To compute $\partial_S \lambda^\star$, one can differentiate \eqref{eq:impl-eq} to obtain $\partial_\smallertext{S} \lambda^\star(z,S)= 1/{\partial_\lambda \phi(z,\lambda)}, \; (z,S) \in \R \times \R_{\smallertext{+}}.$ By previous computation, $\partial_\lambda \phi$ is negative, implying that $\partial_{\smallertext{S}\smallertext{S}} F(z,S) \leq 0$. This shows that $F$ is concave in $S$.

\subsubsection{Green bonds investment}\label{sss:BP2022}

We now study the investment model introduced by \citeayn{baldacci2022governmental}, which differs from the delegated portfolio management example in \Cref{sss:CPT2017} in that it allows a control of arbitrary dimension $m$ and correlated Brownian motions. In this setting, the agent controls the output process through $\beta \coloneqq (\beta_1,\dots,\beta_m)$,
\begin{align*}
	\drm X_t &= (\beta_t^\top v)\drm t + \beta_t^\top \drm W_t,
\end{align*}
where $v \coloneqq (v_1,\dots,v_m)$ denote the drift loadings and $W \coloneqq (W^1,\dots,W^m)$ is an $m$-dimensional Brownian motion with $\drm [W]_t = C\drm t$, for $C \in \Sc_m^{\smallertext{+}}(\mathbb{R})$ such that $C_{ii}=1$ and $|C_{ij}|\le 1$. Using the notations from the previous section and the cost functions from \cite{baldacci2022governmental}, we define
\begin{align*}
	f(z,b) \coloneqq  z\, v^\top b - \frac12 (b-a)^\top K (b-a),\;
	g(S,b) \coloneqq b^\top C b - S, \; (z,S,b) \in \R \times \R_+ \times \R^m,
\end{align*}
where $a \coloneqq (a_1,\dots,a_m)$ are the cost targets, and $K \coloneqq \mathrm{diag}(k_1,\dots,k_m) \succ 0$ the cost weights. The first-order conditions from the associated Lagrangian yield
\begin{align*}
	z v - K (b-a) - 2\lambda C b = 0
	\Longleftrightarrow
	(K + 2\lambda C) b = K a + z v.
\end{align*}
Define $M(\lambda) \coloneqq K + 2\lambda C$ and $r \coloneqq K a + z v$. Then $ b^\star(z,\lambda) = M(\lambda)^{-1} r.$ To ensure $M(\lambda)\succ 0$, set $G \coloneqq K^{-1/2} C K^{-1/2} \succ 0$ and let $\kappa_{\rm max}$ be the largest eigenvalue of $G$. We have
\begin{align*}
	M(\lambda)\succ 0 \Longleftrightarrow K^{-1/2} M(\lambda) K^{-1/2} = \mathrm{I} + 2\lambda G \succ 0
	\Longleftrightarrow 1 + 2\lambda \kappa_i > 0, \;  \forall i\in\{1,\dots,m\}
	\Longleftrightarrow \lambda > -\frac{1}{2\kappa_\smalltext{\rm max}}.
\end{align*}
Hence $M(\lambda)\succ 0$ if and only if $\lambda \in \Lambda\coloneqq\big(-1/(2\kappa_{\rm max}), \infty\big)$. Define $\phi(z,\lambda)\coloneqq b^{\star\top} C\,b^\star= r^\top M(\lambda)^{-1} C\, M(\lambda)^{-1} r,$ and diagonalise $G$. Since $M(\lambda)=K^{1/2}(\mathrm{I}+2\lambda G)K^{1/2}$, we get
$M(\lambda)^{-1}=K^{-1/2}(\mathrm{I}+2\lambda G)^{-1}K^{-1/2}$.
Let $G=UDU^\top$ with $U$ orthonormal and $D=\mathrm{diag}(\kappa_1,\dots,\kappa_m)$.
Set $q\coloneqq K^{-1/2}r$, $\hat q\coloneqq U^\top q$. Then for $\lambda \in \Lambda$
\begin{align*}
	\phi(z,\lambda)= q^\top (\mathrm{I}+2\lambda G)^{-1} G (\mathrm{I}+2\lambda G)^{-1} q = \sum_{i=1}^m \frac{\kappa_i}{(1+2\lambda \kappa_i)^2}\,\hat q_i^2,
	\text{ and } \; \partial_\lambda \phi(z,\lambda)= \sum_{i=1}^m \left(-\frac{4\kappa_i^2}{(1+2\lambda \kappa_i)^3}\right)\hat q_i^2,
\end{align*}
hence $\phi$ is decreasing with respect to $\lambda$ on its domain. Moreover, since $\kappa_i/(1+2\lambda\kappa_i)^2\longrightarrow 0$ as $\lambda\longrightarrow+\infty$ and $\kappa_{\rm max}/(1+2\lambda\kappa_{\rm max})^2\longrightarrow+\infty$ when $\lambda\downarrow -1/(2\kappa_{\rm max})$, we deduce the following limits
\begin{align*}
	\lim_{\lambda\to+\infty}\phi(z,\lambda)=0, \; \text{and} \; \lim_{\lambda\downarrow-1/(2\kappa_\smalltext{\rm max}}\phi(z,\lambda)=+\infty.
\end{align*}
Since $\phi(z,\lambda)$ is continuous and strictly decreasing on $\Lambda$, for any $S>0$ there is a unique solution $\lambda^\star(z,S)$ to $\phi(z,\lambda)=S$. Following the same steps as in the delegated portfolio case, we conclude that $F$ is concave in $S$, verifying \Cref{ass:duality}.

\subsubsection{Demand--response programs in electricity markets }\label{sss:APT2022}

We conclude this section with the study of the demand--response model introduced by \citeayn{aid2022optimal}. In this model, the dynamics of the (one-dimensional) output process is given by
\begin{align*}
	\drm X_t = \sum_{i=1}^n \Big( -\alpha^i_{t} \drm t + \sigma_i \sqrt{\beta^i_{t}} \drm W_{t}^i \Big), \; t \in [0,T],
\end{align*}
with $W \coloneqq (W^1,\dots,W^n)$ a $n$-dimensional Brownian motion and $(\sigma_1,\dots,\sigma_n)\in \big(\R^\star_{+}\big)^n$, and is controlled by the agent through the pair of processes  $(\alpha,\beta)$, taking values in $A \coloneqq [0,\bar{a}]^n$ and $ B \coloneqq (0,1]^n$ respectively, with $\bar{a} >0$. The agent's cost of effort on the drift and volatility are respectively denoted by $c_1(a)$ for $a \in A$ and $c_2(b)$ for $b \in B$; in particular
\begin{align*}
	c_2(b)=\sum_{i=1}^n \frac{\sigma_i^2}{k_i}\big(b_i^{-1}-1\big), \; b \coloneqq (b_1,\dots,b_n)\in (0,1]^n,  \; (k_1,\dots,k_n)\in \big(\R^\star_{+}\big)^n.
\end{align*}
We first note that $x\longmapsto x^{-1}$ is strictly convex on $(0,1]$, so $c_2$ is convex on $B$. For later use, we define the maximal attainable radius $S_{\mathrm{max}}\coloneqq \sum_{i=1}^n\sigma_i^2$, and let for $S\in(0,S_{\max}]$
\begin{align*}
	m(S)\coloneqq \inf_{\{b\in \smallertext{B}: |\sigma(b)|^\smalltext{2}=\smallertext{S}\}}c_2(b),\; \text{with} \;
	|\sigma(b)|^2 \coloneqq \sum_{i=1}^n \sigma_i^2 b_i, \; b \coloneqq (b_1,\dots,b_n)\in (0,1]^n.
\end{align*}
Taking into account the agent's utility of consumption $u(x)$ for $x \in \R$, the agent's `constrained' Hamiltonian is
\begin{align*}
	F(x,z,S) = u(x) - \inf_{a\in \smallertext{A}} \big\{a \cdot \mathbf{1}_n z + c_1(a) \big\}-\frac12 m(S), \; (x,z,S) \in \R \times \R \times (0,S_{\max}].
\end{align*}
We verify below that the mapping $S\longmapsto m(S)$ is convex on $(0,S_{\max}]$, and consequently, $S\longmapsto F(S)$ is concave on the same interval, as required. For any $(S_1,S_2)\in(0,S_{\max}]^2$ and $\theta\in[0,1]$, define $S_\theta\coloneqq \theta S_1+(1-\theta)S_2$. First, compactness of $B$ and continuity of $|\sigma(\cdot)|^2$ ensure the existence of
$b^{(i)}\in B$ such that $|\sigma(b^{(i)})|^2=S_i$ and $c_2(b^{(i)})=m(S_i)$ for $i\in\{1,2\}$. Define then $b^{(\theta)}\coloneqq \theta b^{(1)}+(1-\theta)b^{(2)}\in B$. Because $|\sigma(\cdot)|^2$ is affine in $b$
\begin{align*}
	\big|\sigma\big(b^{(\theta)}\big)\big|^2=\theta\big|\sigma\big(b^{(1)}\big)\big|^2+(1-\theta)\big|\sigma\big(b^{(2)}\big)\big|^2=S_\theta .
\end{align*}
In other words, $b^{(\theta)} \in B$ satisfies the constraint $|\sigma(b^{(\theta)})|^2 = S_\theta$.
Finally, applying convexity of $c_2$, we have
\begin{align*}
	m(S_\theta)\le c_2\big(b^{(\theta)}\big)\le\theta c_2\big(b^{(1)}\big)+(1-\theta)c_2\big(b^{(2)}\big)=\theta m(S_1)+(1-\theta)m(S_2).
\end{align*}
The inequality above proves that $m$ is convex, implying that $F$ is concave in the $S$-variable and \Cref{ass:duality} holds.

\subsection{Duality gap}\label{ss:duality_gap}

{In the three models considered above, \Cref{ass:duality} holds, ensuring that the original contract form \eqref{eq:contract_CPT} is optimal and that the results of these papers remain valid. To the best of our knowledge, other applications of principal--agent problems with volatility control rely on closely related frameworks, as in \citeayn{elie2021mean} which extends \cite{aid2022optimal} with the same form of dynamics and costs: a similar reasoning shows that \Cref{ass:duality} still holds there. While we cannot claim that every applied model necessarily satisfies \Cref{ass:duality}, all those we are aware of actually do. Nevertheless, the assumption is not automatic: beyond usual specifications, one can construct \textit{reasonable} counter-examples where it fails, as highlighted in \cite[Section 4.3]{chiusolo2026new}. Nevertheless, we show here that the new contract form we propose allows to achieve this value.}

\medskip

We start by briefly recalling the framework introduced in \cite[Section 4.3]{chiusolo2026new}. The controlled state equation is given by
\begin{align*}
	\drm X_t = \beta^\P_t \drm W^\P_t, \; t \in [0,T], \; X_0 = 0,
\end{align*}
where $W$ is a one-dimensional Brownian motion and the agent's control process $\beta$ can only take values in $[-1,1]$. The agent's objective function to be maximised over admissible probability measures $\P \in \Pc$ is defined by
\begin{align}\label{eq:agent}
	J_{\smallertext{\rm A}} (\xi,\Sigma,\P) \coloneqq  \E^\P \bigg[ \xi - \int_0^\smallertext{T} c (\beta^\P_t) \drm t \bigg], \; c(b) = 1-b^4, \; b \in [-1,1].
\end{align} 
The principal's criteria is then defined as follows
\begin{align}\label{eq:principal}
	J_{\smallertext{\rm P}} (\xi,\Sigma,\P) \coloneqq  \E^\P \bigg[ X_T - \xi - \int_0^\smallertext{T} c_{\smallertext{\rm P}} \big(\Sigma_t \big) \drm t \bigg], \; c_{\smallertext{\rm P}} (S) \coloneqq  S^3, \; S \in [0,1].
\end{align}
In both criteria, we make explicit the dependence in the density $\Sigma$ of the quadratic variation $[ X ]$, in order to study both the original problem, in which $\Sigma$ is controlled by the agent through his choice of $\beta$, and its `contractible‑volatility' reformulation, in which the principal controls $\Sigma$ thus forcing the agent to choose his control $\beta$ so that $\beta^2 = \Sigma$.

\paragraph*{`Contractible‑volatility' case.} It is shown in \cite{chiusolo2026new} that when the principal directly controls the density $\Sigma$ of the quadratic variation, she can reach the maximum value of $-R_{\smallertext{\rm A}} - 23T/27$. More precisely, in the `contractible‑volatility' formulation, we know that the following contract form is optimal
\begin{align*}
	\xi = y_0 - \int_0^\smallertext{T} F(\Sigma_t) \drm t + \int_0^\smallertext{T} Z_t \drm X_t,
\end{align*}
where the pair of processes $(Z,\Sigma)$ has to be chosen by the principal, and for all $S \in [0,1]$, we have
\begin{align*}
	F(S) \coloneqq  \sup_{\{b \in [\smalltext{-}1,1] : b^\smalltext{2} = \smallertext{S}\}} \big\{ - c(b) \big\}
	= \sup_{\{b \in [\smalltext{-}1,1] : b^\smalltext{2} = \smallertext{S}\}} \big\{ b^4 - 1 \big\}
	= S^2 -1, \; \text{for} \; b^\circ(S) = \pm \sqrt{S}.
\end{align*}
Note that since the constrained Hamiltonian $F$ is convex---not concave---\Cref{ass:duality} is not satisfied. The agent's optimal efforts are given by the maximisers of  $F$, here $\beta_\cdot^\circ = \pm \sqrt{\Sigma_\cdot}$, which leads to $F(\Sigma_\cdot) = \Sigma_\cdot^2 -1$. Finally, one can optimise the principal's criteria defined by \eqref{eq:principal} by choosing $\Sigma \equiv 2/3$, since by pointwise optimisation
\begin{align*}
	V^\circ_{\smallertext{\rm P}}
	= - y_0 - T \bigg(1 - \max_{\smallertext{S} \in [0,1]} \big\{ S^2 - S^3 \big\} \bigg) = -y_0 - 23T/27.
\end{align*}
Maximising this value over $y_0$ under the agent's participation constraint naturally leads to the choice $y_0 = R_{\smallertext{\rm A}}$.

\paragraph*{Duality gap.} It is shown in \cite{chiusolo2026new} that, when restricting to contracts of the form \eqref{eq:contract_CPT} in the original formulation of the principal--agent problem---when $\Sigma$ is not directly controlled by the principal---, the `contractible‑volatility' value previously computed cannot be reached. More precisely, one can first compute that for $\gamma \in \R$
\begin{align*}
	H(\gamma) &\coloneqq  \sup_{b \in [-1,1]} \bigg\{\dfrac12 \gamma b^2 - c(b) \bigg\}
	= \begin{cases}
		\gamma/2,\; \mbox{if}\; \gamma > -2,\; \text{for}\; b^\star(\gamma) \in \{-1,1\}, \\
		-1,\; \mbox{if}\; \gamma = -2,\; \text{for}\; b^\star(\gamma) \in \{-1,0,1\}, \\
		-1,\; \mbox{if}\; \gamma < -2,\; \text{for}\; b^\star(\gamma) =0.
	\end{cases}
\end{align*}
Then, using the form of contract \eqref{eq:contract_CPT} in the principal's criteria \eqref{eq:principal}, where here the density of the quadratic variation is determined through the agent's optimal response to the contract, namely $\Sigma^\star_\cdot = | b^\star(\Gamma_\cdot) |^2$, $t \in [0,T]$, we obtain
\begin{align*}
	J_{\smallertext{\rm P}} \big(\xi,\Sigma^\star,\P^\star \big) = - y_0 + \E^{\P^\smalltext{\star}} \bigg[ \int_0^\smallertext{T} \bigg( H(\Gamma_t) - \dfrac12 \Gamma_t | b^\star(\Gamma_t) |^2  - c_{\smallertext{\rm P}} \big( | b^\star(\Gamma_t) |^2 \big) \bigg)\drm t \bigg],
\end{align*}
where now $\Gamma$ has to be chosen optimally. Again by pointwise optimisation, we compute that for $c_{\smallertext{\rm P}}(S) = S^3$, $S \in [0,1]$
\begin{align*}
	\max_{\gamma \in \R} \bigg\{ H(\gamma) - \dfrac12 \gamma | b^\star(\gamma)|^2  - c_{\smallertext{\rm P}} \big( |b^\star(\gamma)|^2 \big) \bigg\}
	= -1, \; \text{for} \; \gamma^\star = -2,
\end{align*}
leading to a value of $-R_{\smallertext{\rm A}} - T$, (strictly) lower than the one previously obtained in the `contractible‑volatility' case. 

\paragraph*{Reaching the `contractible‑volatility' value.} We show now that by relying on the more general class of contracts \eqref{eq:xi_2BSDE_psi} introduced in this note, the `contractible-volatility' value can indeed be reached, as theoretically proved in \Cref{sss:BSDE_proof}. For this, one can use for instance the contract form \eqref{eq:contract_new_simple}, which becomes in this specific example
\begin{align*}
	\xi = y_0 -  \int_0^\smallertext{T} F \big( \widehat \sigma^2_t \big) \drm t 
	+ \int_0^\smallertext{T} Z_t  \drm X_t + \dfrac12 \int_0^\smallertext{T} \big| \widetilde \Gamma_t -\Gamma_t \big|^2 \drm t
	- \dfrac12 \int_0^\smallertext{T} \big| \widehat \sigma^2_t-\Gamma_t \big|^2 \drm t,
\end{align*}
where $\widetilde \Gamma_t \coloneqq  (\Gamma_t \vee 0) \wedge 1$, for a control process $\Gamma$ chosen by the principal.
Given this contract, the optimal action for the agent is given by the maximiser over $b \in [-1,1]$ of
\begin{align*}
	\widehat F(\gamma, b) &\coloneqq  - c(b) - F\big(b^2 \big) - \dfrac12 \big| b^2 - \gamma \big|^2, \; \text{with} \; F (S) = S^2 - 1, \; S \in [0,1].
\end{align*}
This gives
\begin{align*}
	\max_{b \in [-1,1]} \widehat F(\gamma, b) 
	= - \dfrac12 \min_{b \in [-1,1]} \big| b^2 - \gamma \big|^2
	= - \dfrac12 \big| (\gamma \vee 0) \wedge 1 - \gamma \big|^2, 
	\; \text{for} \; b^\star(\gamma) = \pm \sqrt{(\gamma \vee 0) \wedge 1}.
\end{align*}
Therefore, under the agent's optimal response, we have $\widehat \sigma^2 = |b^\star(\Gamma)|^2 = \widetilde \Gamma$, and the contract becomes
\begin{align*}
	\xi 
	&= y_0 + T - \int_0^\smallertext{T} |b^\star(\Gamma_t)|^4 \drm t + \int_0^\smallertext{T} Z_t  \drm X_t.
\end{align*}
Hence, we can compute the principal's objective under the agent's best response,
\begin{align*}
	J_{\smallertext{\rm P}}(\xi,\widehat \sigma^2,\P^\star) &= \E^{\P^\smalltext{\star}} \bigg[ X_T - \xi - \int_0^\smallertext{T} c_{\smallertext{\rm P}}(\widehat \sigma^2_t) \drm t \bigg] 
	= - y_0 - T + \E^{\P^\smalltext{\star}} \bigg[ \int_0^\smallertext{T} \Big( |b^\star(\Gamma_t)|^4 - c_{\smallertext{\rm P}} \big( |b^\star(\Gamma_t)|^2 \big) \Big) \drm t \bigg],
\end{align*}
assuming sufficient integrability condition. Recalling that $c_{\smallertext{\rm P}} (S) = S^3$, by pointwise optimisation we have
\begin{align*}
	\max_{\gamma \in \R} \Big\{ |b^\star(\gamma)|^4 - c_{\smallertext{\rm P}} \big( |b^\star(\gamma)|^2 \big) \Big\}
	= \max_{\gamma \in \R} \Big\{ |(\gamma \vee 0) \wedge 1|^2 - |(\gamma \vee 0) \wedge 1|^3 \Big\}
	= \max_{\gamma \in [0,1]} \{ \gamma^2 - \gamma^3 \}
	= \dfrac4{27}, \; \text{for} \; \gamma^\star = \dfrac23.
\end{align*}
We thus conclude that the new contract form \eqref{eq:contract_new_simple} allows the principal to achieve her value in the  reformulated problem. To achieve the `contractible‑volatility' value, one can also use the contract form \eqref{eq:contract_Dylan}. For this, we compute
\begin{align*}
	F^{\star \star}(S) \coloneqq  \inf_{\gamma \in \R} \bigg\{H(\gamma) - \dfrac12 \gamma S \bigg\}
	= S-1, \; \text{for} \; \gamma^\star = -2.
\end{align*}
Note that this confirms that \Cref{ass:duality} is not satisfied here. Indeed,
$F^{\star \star} (S) - F_{\smallertext{\rm A}}(S) = S - S^2$, which is (strictly) positive on $(0,1)$, and equal to $0$ only if $S \in \{0,1\}$. Therefore, $F^{\star \star}(S) \neq F(S)$ for all $S \in (0,1)$. 
Considering the contract form \eqref{eq:contract_Dylan}, we obtain here
\begin{align*}
	\xi = y_0 - \int_0^\smallertext{T} \bigg( H (\Gamma_t) + \big| \widehat \sigma^2_t \big|^2 - \widehat \sigma^2_t - \dfrac12 \Gamma_t \widehat \sigma^2_t \bigg) \drm t
	+ \int_0^\smallertext{T} Z_t  \drm X_t.
\end{align*}
Plugging this form into the agent's expected utility, we obtain
\begin{align*}
	J_{\smallertext{\rm A}} (\xi,\widehat \sigma^2,\P) = y_0 + \E^\P\bigg[\int_0^\smallertext{T} \bigg(|\beta^\P_s|^2 + \dfrac12 \Gamma_s |\beta^\P_s|^2 - 1 - H(\Gamma_s) \bigg) \drm s  \bigg],
\end{align*}
and therefore the optimal effort is obtained by computing
\begin{align*}
	\max_{b \in [-1,1]} \bigg\{ |b|^2 + \dfrac12 \gamma |b|^2\bigg\}
	&= \begin{cases}
		1+\gamma/2,\; \mbox{if}\; \gamma > -2, \;\text{for}\; \widetilde b(\gamma) \in \{-1,1\}, \\
		0, \;\mbox{if}\; \gamma = -2, \;\text{for}\; \widetilde b(\gamma) \in [-1,1], \\
		0,\; \mbox{ if}\; \gamma < -2, \;\text{for}\; \widetilde b(\gamma) =0.
	\end{cases}
\end{align*}
Looking at the principal's problem defined by \eqref{eq:principal} when using contracts of the form \eqref{eq:contract_Dylan}, we have
\begin{align*}
	J_{\smallertext{\rm P}} \big(\xi,\widehat \sigma^2,\P^\star \big) 
	&= - y_0 + \E^{\P^\smalltext{\star}} \bigg[ \int_0^\smallertext{T} \bigg( H (\Gamma_t) + | \widetilde b(\Gamma_t) |^4 - | \widetilde b(\Gamma_t) |^2 - \dfrac12 \Gamma_t | \widetilde b(\Gamma_t) |^2 
	- c_{\smallertext{\rm P}} \big( | \widetilde b(\Gamma_t) |^2 \big) \bigg) \drm t \bigg],
\end{align*}
where $\Gamma$ has to be chosen optimally. Again by pointwise optimisation, we obtain for $c_{\smallertext{\rm P}}(S) = S^3$
\begin{align*}
	\Gamma_t^\star \in \argmax_{\gamma \in \R} \bigg\{ H (\gamma) + | \widetilde b(\gamma)|^4 - | \widetilde b(\gamma) |^2 - \dfrac12 \gamma | \widetilde b(\gamma) |^2 
	- c_{\smallertext{\rm P}} \big( | \widetilde b(\gamma) |^2 \big) \bigg\} \equiv -2, \; t \in [0,T].
\end{align*}
Indeed, such a choice of $\gamma = -2$ makes the agent indifferent between any effort taking value in $[-1,1]$, and gives a value of $-y_0 - T (1 + | \widetilde b(\gamma) |^4 - | \widetilde b(\gamma) |^6)$ for the principal. She can then maximise this value by choosing among agent's best response $\widetilde b(\gamma) \in \{-\sqrt{2/3}, \sqrt{2/3} \}$, and thus achieve her `contractible‑volatility' value.

\subsection{Another counter-example}\label{ss:other}

In this subsection, we exhibit a counter-example within the framework of \cite[Section 3]{cvitanic2017moral}. In particular, despite additional convexity assumptions compared to \cite{cvitanic2018dynamic}, this more restrictive model may still fail to satisfy \Cref{ass:duality}. Let $(\alpha,\beta)$ be the agent's effort taking values in $\R^2$, and consider the following dynamics for the output process
\begin{align*}
	\drm X_t = \beta^\P_t( \alpha^\P_t \drm t+\drm W_t), \; t \in [0,T],
\end{align*}
for $W$ a one-dimensional Brownian motion. The agent's cost of effort is further defined as $k(a,b) = k_1 |b| + k_2 a^{2}$ for $(a,b)\in \R^2$ and fixed constants $k_1>0$ and $k_2>0$. Note that the above specification satisfies the conditions required in \cite[Section 3]{cvitanic2017moral}: the cost function $k$ is non-negative and convex and the linear growth condition \cite[Equation (6)]{cvitanic2017moral} holds since for every $(a,b)\in \R^{2}$, $|a|+|b|\le\ C(1+|a|+|b|)$ for all $C \geq 1$.

\medskip

Within this framework, the agent's constrained Hamiltonian can be defined for $(z,S) \in \R \times \R_{\smallertext{+}}$ as
\begin{align*}
	F(z, S)&\coloneqq \sup_{\{(a,b)\in \R^\smalltext{2} : b^\smalltext{2}=\smallertext{S}\}}\big\{z a b - k_1 |b| - k_2 a^{2} \big\}
	= \sup_{\{b\in \mathbb R:  b^\smalltext{2}=\smallertext{S}\}} \bigg\{-k_1 |b| + \sup_{a \in \mathbb R}\big\{z a b -k_2 a^2\big\}\bigg\}.
\end{align*}
Standard computations give
\begin{align*}
	F(z, S) = - k_1 \sqrt{S} + \dfrac{z^2 S}{4 k_2}, \; \text{with corresponding efforts} \; b^\star(z,S) \coloneqq \pm \sqrt{S}, \; a^\star(z,S) \coloneqq \dfrac{z b^\star(z,S) }{2 k_2}, \; (z,S) \in \R \times \R_{\smallertext{+}}.
\end{align*}
Differentiating $F$ with respect to $S > 0$, we obtain for all $z \in \R$
\begin{align*}
	\partial_\smallertext{S} F(z,S)=\frac{z^{2}}{4 k_2}-\frac{k_1}{2\sqrt S},\;
	\partial_{\smallertext{S}\smallertext{S}} F(z,S)=\frac{k_1}{4 S^{3/2}}.
\end{align*}
Recalling that $k_1$ is a positive constant, the second derivative is positive for all $S > 0$. In other words, $F$ is strictly convex with respect to the $S$-variable on $(0,\infty)$, implying that \Cref{ass:duality} is not satisfied in this context.

\appendix
\section{Appendix}\label{sec:app}
\begin{proof}[Proof of Proposition \ref{prop:agent}]
Let $\psi \in \Oc $, $y_0 \geq R_{\smallertext{\rm A}}$ and $(Z,\Gamma) \in \Vc^\psi$, and consider the associated contract $\xi \in \Cc$ as in the statement of the proposition. First, the integrability condition on $Y^{y_{\smalltext{0}},\smallertext{Z},\smallertext{\Gamma},\psi}$ in \Cref{def:contrat_new} directly implies that the terminal payment $\xi$ satisfies the integrability condition \eqref{eq:integrability_contract_agent}. To ensure that $\xi \in \Xi$, it therefore suffices to verify that the agent's participation constraint is satisfied.
	With this in mind, letting $\P \in \Pc$, we can use It\^o's formula to compute
	\begin{align*}
		\Kc_\smallertext{T}^\P Y_\smallertext{T}^{y_{\smalltext{0}},\smallertext{Z},\smallertext{\Gamma},\psi} 
		&= y_0 
		- \int_0^\smallertext{T} \Kc_s^\P \Big( \psi_s \big(Y^{y_{\smalltext{0}},\smallertext{Z},\smallertext{\Gamma},\psi}_s, Z_s,\Gamma_s,\widehat \sigma^2_s \big) + k_s \big(\beta_s^\P \big) Y_s - Z_s \cdot \mu_s(\beta^\P_s) \Big) \drm s+ \int_0^\smallertext{T} \Kc_s^\P Z_s \cdot \sigma_s (\beta^\P_s) \drm W^\P_s.
	\end{align*}

	Plugging this into the agent's utility defined in \eqref{eq:pb_agent} and using the integrability conditions on $Z$, we obtain
	\begin{align*}
		J_{\smallertext{\rm A}} (\xi,\P )
		&= y_0 + \E^\P \bigg[
		\int_0^\smallertext{T} \Kc_s^\P \Big( f_s \big(Y^{y_{\smalltext{0}},\smallertext{Z},\smallertext{\Gamma},\psi}_s,Z_s,\beta^\P_s \big) - \psi_s \big(Y^{y_{\smalltext{0}},\smallertext{Z},\smallertext{\Gamma},\psi}_s, Z_s,\Gamma_s,\widehat \sigma^2_s \big) \Big) \drm s \bigg],
	\end{align*}
	recalling that $f$ is defined in \eqref{eq:hamiltonian_constrained}. Using \Cref{cond:revealing} and the definition of $F$ in \eqref{eq:hamiltonian_constrained}, we notice that
	\begin{align}\label{eq:optimal_effort}
		\sup_{ b \in \smallertext{B}} \big\{ f_t (x,y,z,b)-\psi_t(x,y,z,\gamma,S) \big\} = \sup_{\smallertext{S} \in \smallertext{{\bf \Sigma}}_\smalltext{t}(x)} \bigg\{ \sup_{b \in \smallertext{B}_\smalltext{t}(x,\smallertext{S})} \big\{ f_t(x,y,z,b) \big\} - \psi_t(x,y,z,\gamma,S)  \bigg\} = 0,
	\end{align}
	for all $(t,x,y,z,\gamma) \in [0,T] \times \Omega \times \R \times \R^d \times E$, implying both that $J_{\smallertext{\rm A}} (\xi,\P ) \leq y_0$ for all $\P \in \Pc$ and $V_{\smallertext{\rm A}} (\xi) = y_0$.
	
	\medskip
	It remains to show that there exists (at least) an optimal response for the agent to $\xi$. This can be deduced from \eqref{eq:def-optimal} in \Cref{def:contrat_new}, which states that there is $\P^\star \in \Pc$ such that
	\begin{align*}
		\psi_t \big(Y^{y_\smalltext{0},\smallertext{Z},\smallertext{\Gamma},\psi}_t,Z_t,\Gamma_t,\widehat \sigma^2_t\big) = F_t \big(Y^{y_\smalltext{0},\smallertext{Z},\smallertext{\Gamma},\psi}_t,Z_t,\widehat \sigma^2_t\big) = f_t \big(Y_t^{y_\smalltext{0}, \smallertext{Z},\smallertext{\Gamma},\psi}, Z_t, \beta^{\P^{\smalltext{\star}}}_t \big), \drm t \otimes \P^\star \text{\rm --a.e. on } [0,T] \times \Omega,
	\end{align*}
	implying both the existence of a process $\mathfrak{S}^{\star,\psi}_\cdot \coloneqq S^{\star,\psi}_\cdot(Y^{y_\smalltext{0},\smallertext{Z},\smallertext{\Gamma},\psi}_\cdot,Z_\cdot,\Gamma_\cdot)$, where $S^{\star,\psi}$ satisfies
	\begin{align*}
		S^{\star,\psi}_t(x,y,z,\gamma) \in \argmin_{\smallertext{S} \in \smallertext{{\bf \Sigma}}_\smalltext{t}(x)} \big\{ \psi_t(x,y,z,\gamma,S) - F_t(x,y,z,S) \big\} , \; (t,x,y,z,\gamma) \in [0,T] \times \Omega \times \R \times \R^d \times E,
	\end{align*}
	and of a process $\mathfrak{b}^{\star,\psi}_\cdot \coloneqq b^{\star,\psi}_\cdot(Y^{y_\smalltext{0},\smallertext{Z},\smallertext{\Gamma},\psi}_\cdot,Z_\cdot,\Gamma_\cdot)$, with $b^{\star,\psi}$ defined by \eqref{eq:optimum}. Hence $J_{\smallertext{\rm A}} (\xi,\P ) \leq y_0$ for all $\P \in \Pc$, with $J_{\smallertext{\rm A}} (\xi,\P ) = y_0$ if and only if the agent's effort is set to $\mathfrak{b}^\star_t = b^\star_t(Y^{y_\smalltext{0},\smallertext{Z},\smallertext{\Gamma},\psi}_t,Z_t,\Gamma_t)$,  $\mathrm{d}t\otimes\P$--a.e., where $b^\star$ satisfies \eqref{eq:optimum}. 
\end{proof}

{\small\bibliography{bibliographyDylan}}

\end{document}